\documentclass{gtpart}
\usepackage{pinlabel}

\usepackage{amscd,amssymb, amsmath,amsfonts, wasysym, mathrsfs, enumerate, mathtools,hhline,color}
\usepackage{graphicx}
\usepackage[all, cmtip]{xy}

\usepackage{url}

\author[LQ]{Lizhen Qin}
\givenname{Lizhen}
\surname{Qin}
\address{Department of Mathematics\\
Nanjing University \\\newline
22 Hankou Road\\Nanjing\\Jiangsu 210093\\P.R.China}
\email{qinlz@nju.edu.cn}

\author[BW]{Botong Wang}
\givenname{Botong}
\surname{Wang}
\address{Department of Mathematics\\University of Wisconsin-Madison\\\newline 
480 Lincoln Drive\\Madison\\ WI 53706\\ USA}
\email{bwang274@wisc.edu}



\arxivreference{1601.04337}
\arxivpassword{rxjkw}

\theoremstyle{definition}
\newtheorem{theorem}{Theorem}[section]
\newtheorem{prop}[theorem]{Proposition}

\newtheorem{lemma}[theorem]{Lemma}

\theoremstyle{definition}

\newtheorem{defn}[theorem]{Definition}

\newtheorem{rmk}[theorem]{Remark}

\newtheorem{ex}[theorem]{Example}
\newtheorem*{ex*}{Example}

\newtheorem*{question}{Question}

\numberwithin{equation}{section}

\DeclareMathOperator{\Char}{Char}                

\DeclareMathOperator{\homo}{Hom}

\DeclareMathOperator{\tr}{tr}
\DeclareMathOperator{\Id}{Id}
\DeclareMathOperator{\GL}{GL}
\DeclareMathOperator{\SL}{SL}
\DeclareMathOperator{\Sp}{Sp}
\DeclareMathOperator{\re}{Re}

\DeclareMathOperator{\kn}{Ker}
\DeclareMathOperator{\ckn}{Coker}
\DeclareMathOperator{\im}{Im}
\DeclareMathOperator{\Int}{Int}
\DeclareMathOperator{\PD}{PD}

\title[Complex and Symplectic Manifolds that are Nonk\"{a}hler]{A Family of Compact Complex and Symplectic Calabi-Yau Manifolds that are Nonk\"{a}hler}

\begin{document}

\begin{abstract}
We construct a family of $6$-dimensional compact manifolds $M(A)$ which are simultaneously diffeomorphic to complex Calabi-Yau manifolds and symplectic Calabi-Yau manifolds. They have fundamental groups $\mathbb{Z} \oplus \mathbb{Z}$, their odd-degree Betti numbers are even, they satisfy the hard Lefschetz property, and their real homotopy types are formal. However, $M(A) \times Y$ are not homotopy equivalent to any compact K\"{a}hler manifold for any topological space $Y$. The main ingredient to show the nonk\"ahlerness is a structure theorem of cohomology jump loci due to the second author.
\end{abstract}

\maketitle

\section{Introduction}
A K\"ahler manifold is a symplectic manifold together with a compatible complex structure. People have been interested in searching examples of compact complex or compact symplectic manifolds which are nonk\"{a}hler. In the 1940s, Hopf constructed complex manifolds $S^{1} \times S^{2N-1}$ ($N>1$) that are nonk\"{a}hler (\cite[p.\ 14]{morrow_kodaira}) because their second Betti numbers vanish. By similar reasons, it is easy to construct many compact complex nonk\"{a}hler manifolds. On the other hand, it is more difficult and hence more interesting to find the symplectic counterparts. The existence of compact symplectic nonk\"{a}hler manifolds had been an open question for many years until Thurston discovered the first example in \cite{thurston}. From then on, there has been many works on the compact symplectic nonk\"ahler manifolds (e.g. \cite{babenko_taimanov}, \cite{cordero_fernandez_gray}, \cite{fernandez_leon_saralegui}, \cite{Fernandez_Mumoz2}, \cite{gompf}, \cite{mcduff}, \cite{park}, and \cite{tralle_oprea}). Motivated by physics, it is particularly interesting to construct complex or symplectic nonk\"{a}hler manifolds satisfying the Calabi-Yau property (e.g. \cite{akhmedov}, \cite{baldridge_kirk}, \cite{bogomolov}, \cite{fine_panov1}, \cite{fine_panov2}, \cite{friedman}, \cite{goldstein_prokushkin}, \cite{grantcharov}, \cite{guan1}, \cite{guan2}, \cite{gutowski_ivanov_papadopoulos}, \cite{lu_tian}, \cite{smith_thomas_yau}, \cite{torres_yazinski}, and \cite{tseng_yau}).

In this paper, we will construct a family of $6$-dimensional compact smooth manifolds which are simultaneously diffeomorphic to complex and symplectic manifolds. They are Calabi-Yau with respect to both the complex structure and the symplectic structure.  We call these manifolds complex and symplectic Calabi-Yau. Notice that our notion of ``complex and symplectic" is different from the notion ``holomorphic symplectic" (see \cite{guan1}). A holomorphic symplectic manifold is a complex manifold endowed with a holomorphic symplectic form. A holomorphic symplectic manifold certainly carries both complex and symplectic structures and is Calabi-Yau in either sense. However, such manifolds are always of real dimension $4n$, and hence never of dimension 6.

The examples in this paper are interesting to us because they share many properties with compact K\"{a}hler manifolds. However, they are nonk\"ahler in a strong sense. Before stating the main result, we want to make the notion ``nonk\"ahler" precise. For a manifold $M$, we can interpret nonk\"{a}hlerness in one of the following ways.
\begin{quote}
The manifold $M$ is (i) not homotopy equivalent (ii) not homeomorphic (iii) not diffeomorphic (iv) not biholomorphic (v) not symplectomorphic to any compact K\"{a}hler manifold.
\end{quote}
Here, (iv) also means that $M$ does not carry a K\"{a}hler structure compatible with its complex structure, and (v) also means that $M$ does not carry a K\"{a}hler structure compatible with its symplectic structure. Clearly, (i)$\Rightarrow$(ii)$\Rightarrow$(iii)$\Rightarrow$(iv) and (v). The nonk\"{a}hlerness of our examples is in a sense even stronger than (i).

Let us focus on the nonk\"{a}hlerness in the sense of (i) for now. Many works to prove a manifold being nonk\"ahler in the sense of (i) are to show the manifold does not satisfy one of the following properties.
\begin{enumerate}[\hspace{0.6cm}(a)]
\item Fundamental Groups: Only a small class of groups, named K\"{a}hler groups, can be realized as the fundamental groups of compact K\"{a}hler manifolds. (See e.g. \cite{A_B_C_K_T} and \cite{burger} for a survey.)

\item Betti Numbers: The odd-degree Betti numbers of compact K\"{a}hler manifolds are even (\cite[p.\ 117]{griffiths_harris}).

\item Hard Lefschetz Property: For a compact K\"{a}hler manifold $X$ of real dimension $2n$, there exists $\alpha \in H^{2}(X; \mathbb{R})$ such that, for each $j$,
\[
L_{\alpha}^{j}: H^{n-j}(X; \mathbb{R}) \overset{\alpha^{j} \cup}{\longrightarrow} H^{n+j}(X; \mathbb{R}),
\]
is an isomorphism of cohomology groups. In fact, one can choose $\alpha$ to be the cohomology class represented by the K\"{a}hler form of $X$ (\cite[p.\ 122]{griffiths_harris}).

\item Formality: The real homotopy type of a compact K\"{a}hler manifold is formal (\cite{D_G_M_S}).
\end{enumerate}

Now let us state our main theorem.
\begin{theorem}\label{theorem_complex-symplectic}
For each $A \in \SL(2, \mathbb{Z}[\sqrt{-1}])$ such that the absolute value of its trace is $|\tr (A)| > 2$, there exists a $6$-dimensional compact smooth manifold $M(A)$ (constructed in Proposition \ref{proposition_M(A)}) which satisfies the following properties.
\begin{enumerate}
\item The manifold $M(A)$ is simultaneously $C^{\infty}$ diffeomorphic to a complex manifold $M_{C}$ and a smooth symplectic manifold $M_{S}$.

\item The complex manifold $M_{C}$ is holomorphically Calabi-Yau, i.e. the canonical line bundle of $M_{C}$ is trivial as a holomorphic line bundle.

\item The symplectic manifold $M_{S}$ is integrally Calabi-Yau, i.e. $c_{1} (M_{S}) = 0$. Here $c_{1} (M_{S}) \in H^{2}(M_{S}; \mathbb{Z})$ is the first Chern class of $M_{S}$ with respect to its symplectic structure.

\item The fundamental group $\pi_{1} (M(A)) \cong \mathbb{Z} \oplus \mathbb{Z}$. Hence it is a K\"{a}hler group.

\item All odd-degree Betti numbers of $M(A)$ are even.

\item The following map is an isomorphism for each $j$,
\[
L_{[\omega]}^{j}: H^{3-j}(M_{S}; \mathbb{R}) \overset{[\omega]^{j} \cup}{\longrightarrow} H^{3+j}(M_{S}; \mathbb{R}),
\]
where $[\omega] \in H^{2}(M_{S}; \mathbb{R})$ is the cohomology class represented by the symplectic form $\omega$ of $M_{S}$, and $L_{[\omega]}^{j} (\beta) = [\omega]^{j} \cup \beta$. Hence $M(A)$ satisfies the hard Lefschetz property.

\item The real homotopy type of $M(A)$ is formal.

\item However, given any topological space $Y$, the product $M(A) \times Y$ is not homotopy equivalent to any compact K\"{a}hler manifold.
\end{enumerate}
\end{theorem}

The properties (4), (5), (6) and (7) in Theorem \ref{theorem_complex-symplectic} show that $M(A)$ meets the above criteria (a), (b), (c) and (d). Furthermore, (1) shows $M(A)$ carries both the complex and symplectic structures. Nevertheless, property  (8) shows that $M(A)$ is highly nonk\"{a}hler. Putting $Y$ to be a point, then (8)  implies that $M(A)$ is not homotopy equivalent to any compact K\"ahler manifold.

Another advantage of (8) is that, by taking products of $M(A)$ with other manifolds, we get many nonk\"{a}hler examples of higher dimensions. For example, let $K$ be a Kummer surface (see Definition \ref{definition_kummer}). Then $M(A) \times \prod_{j=1}^{n} K$ are manifolds satisfying all conclusions of Theorem \ref{theorem_complex-symplectic} except that their dimensions are $4n+6$.

Notice that Calabi-Yau properties do depend on the choice of complex and symplectic structures. In general, the conclusions (2) and (3) in Theorem \ref{theorem_complex-symplectic} are independent. So both of them are listed in the statement.

There are many matrices $A$ satisfying the assumption of Theorem \ref{theorem_complex-symplectic} (see Example \ref{example_complex-symplectic}). Moreover, $M(A)$ will indeed form a large family of homotopy types thanks to the following theorem.
\begin{theorem}\label{theorem_not_homotopy}
Given two manifolds $M(A_{1})$ and $M(A_{2})$ as in Theorem \ref{theorem_complex-symplectic}, if $A_{1}$ and $A_{2}$ have different spectral radii, then $M(A_{1})$ is not homotopy equivalent to $M(A_{2})$.
\end{theorem}

Our new ingredient to show nonk\"ahlerness is from Hodge theory with twisted coefficients. More precisely, we prove the following theorem, which is essentially a consequence of the structure theorem of cohomology jump loci of compact K\"ahler manifolds in \cite{wang}.
\begin{theorem}\label{theorem_monodromy}
Let $p: E\to S^1$ be a fiber bundle with path-connected fiber $F$. Suppose $H^{j}(F; \mathbb{C})$ are finite dimensional for all $j$ and suppose $E$ is homotopy equivalent to a compact K\"{a}hler manifold. Then the eigenvalues of the monodromy action on $H^{j}(F; \mathbb{C})$ are roots of unity for every $j$.
\end{theorem}

The construction of $M(A)$ is motivated by Thurston's example mentioned above (\cite{thurston}). We sketch the construction here and defer more details to Section \ref{section_main}. Each $A \in \SL(2, \mathbb{Z}[\sqrt{-1}])$ yields a biholomorphic automorphism $A_{K}: K \rightarrow K$ of a Kummer surface (Proposition \ref{proposition_biholomorphic}). Gluing the two ends of the cylinder $S^{1} \times [0,1] \times K$ by the relation $(a,0,x) \sim (a, 1, A_{K}(x))$, we get $M(A) = (S^{1} \times [0,1] \times K) / \sim$. Note that $K$ has a symplectic structure resulting from a holomorphic symplectic form on it. As $A_{K}$ preserves the holomorphic symplectic form, $M(A)$ is both complex Calabi-Yau and symplectic Calabi-Yau. This manifold $M(A)$ is a fiber bundle over $S^{1}$ with fiber $S^{1} \times K$. More generally, given any topological space $Y$, the product $M(A) \times Y$ is a fiber bundle over $S^1$ with fiber $S^{1} \times K \times Y$. We will relate the monodromy actions on $H^{j}(S^1\times K\times Y; \mathbb{C})$ to the matrix $A$. As far as $A$ satisfies the assumption in Theorem \ref{theorem_complex-symplectic}, the monodromy actions will violate the conclusion of Theorem \ref{theorem_monodromy}.

We would like to point out that our examples are similar to the ones in \cite{magnusson} by Magn\'{u}sson, where\ a weaker result was proved: those examples are not biholomorphic to any compact K\"ahler manifold.

The outline of this paper is as follows. In Section \ref{section_kummer}, we recall the definition and some properties of Kummer surfaces.  In Sections \ref{section_main} and \ref{section_formality}, we construct examples $M(A)$ and prove the main Theorem \ref{theorem_complex-symplectic}. Finally, Theorems \ref{theorem_not_homotopy} and \ref{theorem_monodromy} are proved in Section \ref{section_jump}.

\textbf{Acknowledgement.} We thank the anonymous referee
for informing us the related work \cite{magnusson}, and for the helpful comments which lead to a simpler proof of Proposition \ref{proposition_symplectic_CY}. We are also grateful for various discussions from June Huh, Stefan Papadima, Alex Suciu and Claire Voisin. We thank Luis Saumell for careful reading of an earlier version of the paper and his helpful suggestions.

\section{Kummer Surfaces}\label{section_kummer}
In this section, we will study a Kummer surface which plays an important role in this paper. For simplicity, we will only use the Kummer surface defined by the stardard lattice. Let us recall its definition first.

Let $\Lambda = \{ (a_1+a_2\sqrt{-1}, a_3+a_4\sqrt{-1}) \mid a_{i} \in \mathbb{Z} \}$ be the standard lattice of $\mathbb{C}^{2}$. Then $\mathbb{C}^{2} / \Lambda$ is a torus of complex dimension $2$, which we denote by $T$. The universal covering map $\pi: \mathbb{C}^{2} \rightarrow T$ is also a homomorphism of complex Lie groups. Let
\[
\hat{\Lambda} = \frac{1}{2}\Lambda=\left\{ \left( \frac{a_1}{2}+\frac{a_2}{2}\sqrt{-1}, \frac{a_3}{2}+\frac{a_4}{2}\sqrt{-1} \right) \middle| \; a_i \in \mathbb{Z}. \right\}
\]

Then $\pi (\hat{\Lambda})$ consists of $16$ points $w_{j}$ ($j=1, \dots, 16$) in $T$. Let $w_{1} = \pi (0)$.

Let $C_{2}=\{ \pm1 \}$ be the group with two elements. It acts on $\mathbb{C}^{2}$ by multiplication. This action descends to a holomorphic action on $T$ with fixed points exactly $w_{j}$ ($j=1, \dots, 16$). Denote the quotient $T/C_2$ by $\bar{T}$. Then $\bar{T}$ is a complex orbifold with $16$ singularities $[w_{j}]$, where $[w_{j}]$ is the image of $w_{j}$ in the quotient space. We shall resolve these singularities and obtain a smooth complex surface.

Here we give a concrete description of the blowup map. Locally, each singularity in the orbifold is isomorphic to the singularity in the quotient orbifold $\mathbb{C}^{2}/ C_{2}$ at the singular point $[0]$. Denote by $\mathcal{H}$ the Hopf line bundle over $\mathbb{CP}^{1}$, i.e.
\[
\mathcal{H} = \{ (v,l) \in \mathbb{C}^{2} \times \mathbb{CP}^{1} \mid v \in l \}.
\]
Define a holomorphic map $F_{1}: \mathcal{H} \rightarrow \mathbb{C}^{2} / C_{2}$ by $F_{1}(v,l) = [v]$, where $[v]$ is the image of $v$ in $\mathbb{C}^{2} / C_{2}$. By identifying the zero section of $\mathcal{H}$ with $\mathbb{CP}^{1}$, we have $F_{1}(\mathbb{CP}^{1}) = [0]$ and $F_{1}|_{\mathcal{H} \setminus \mathbb{CP}^{1}}$ is a holomorphic double covering of $(\mathbb{C}^{2} / C_{2}) \setminus \{ [0] \}$. Let $\mathcal{H}^{2} = \mathcal{H} \otimes \mathcal{H}$ be the tensor square of $\mathcal{H}$. We define another holomorphic map $F_{2}: \mathcal{H} \rightarrow \mathcal{H}^{2}$ by $F_{2} (v,l) = (v \otimes v, l)$. Also by identifying the zero section of $\mathcal{H}^2$ with $\mathbb{CP}^1$, we have $F_{2} (\mathbb{CP}^{1}) = \mathbb{CP}^{1}$ and $F_{2}|_{\mathcal{H} \setminus \mathbb{CP}^{1}}: \mathcal{H} \setminus \mathbb{CP}^{1} \rightarrow \mathcal{H}^{2} \setminus \mathbb{CP}^{1}$ is a holomorphic double covering. It is straightforward to check that $F = F_{1} \circ F_{2}^{-1}: \mathcal{H}^{2} \rightarrow \mathbb{C}^{2} / C_{2}$ is a well-defined holomorphic map. Moreover, $F(\mathbb{CP}^{1}) = \{ [0] \}$ and $F|_{\mathcal{H}^{2} \setminus \mathbb{CP}^{1}}: \mathcal{H} ^{2} \setminus \mathbb{CP}^{1} \rightarrow (\mathbb{C}^{2} / C_{2}) \setminus \{ [0] \}$ is a biholomorphic map. Thus $F: \mathcal{H}^{2} \rightarrow \mathbb{C}^{2} / C_{2}$ resolves the singularity $[0]$ of $\mathbb{C}^{2} / C_{2}$. The map $F$ is called the blowup of $\mathbb{C}^{2} / C_{2}$ at $[0]$.

Now we come back to $\bar{T}$. By blowing up all singular points $\{[w_j]\}$ in  $\bar{T}$, we obtain a resolution of singularity $P: K \rightarrow \bar{T}$.

\begin{defn}\label{definition_kummer}
The complex surface $K$ constructed above is called the Kummer surface defined by the standard lattice.
\end{defn}

Immediately, we have the following lemma.

\begin{lemma}\label{lemma_kummer}
The map $P: K \rightarrow \bar{T}$ is continuous. Each $Y_{j} = P^{-1} ([w_{j}])$ is a closed holomorphic submanifold of $K$. Furthermore,
\[
P|_{K \setminus \bigsqcup_{1\leq j\leq 16} Y_{j}}:  K \setminus \bigsqcup_{1\leq j\leq 16} Y_{j} \rightarrow \bar{T} \setminus \bigsqcup_{1\leq j\leq 16} \{ [w_{j}] \}
\]
is a biholomorphic map.
\end{lemma}
Each $Y_{j}$ defined as above is called an exceptional divisor in $K$.

\begin{rmk}\label{holomorphic_symplectic_form}
It is well known that every Kummer surface $K$ is Calabi-Yau (see \cite[p.\ 241]{Barth_Peters_Van de Ven}), which means that its canonical line bundle $\mathcal{K}_{K} = \bigwedge^{2} \left( T^{*} K \right)$ is holomorphically trivial. Hence there exist holomorphic symplectic forms, i.e. non-degenerate holomorphic $2$-forms, on $K$. In fact, we can give an explicit description of such a form. Let $(z_{1}, z_{2})$ be the standard coordinate of $\mathbb{C}^{2}$. Then $dz_{1} \wedge dz_{2}$ is a holomorphic symplectic form on $T$. Since $dz_{1} \wedge dz_{2}$ is invariant under the $C_{2}$ action, it descends to a holomorphic form on the regular part of $\bar{T}$, which we also denote by $dz_{1} \wedge dz_{2}$. One can easily check that $dz_{1} \wedge dz_{2}$ extends to a non-degenerate holomorphic form on $K$. Denote this holomorphic symplectic form by $\varpi$, and denote its real part by $\re \varpi$. Then $K$ is a real symplectic manifold with symplectic form $\re \varpi$.
\end{rmk}

\begin{rmk}
A Kummer surface $K$ carries many real symplectic forms compatible with its natural complex structure, which define K\"{a}hler structures on $K$. However, $\re \varpi$ is not such a form. We shall exclusively consider the symplectic structure $(K, \re \varpi)$ of $K$ throughout this paper.
\end{rmk}

\begin{prop}\label{proposition_kummer_SCY}
The symplectic manifold $(K, \re \varpi)$ is integrally Calabi-Yau. In other words, $c_{1} (K, \re \varpi)= 0$, where $c_{1}(K, \re \varpi) \in H^{2}(K; \mathbb{Z})$ is the first Chern class of $K$ with respect to the symplectic structure $\re \varpi$.
\end{prop}
\begin{proof}
Since $\varpi$ is holomorphic symplectic, the structural group $\Sp (4, \mathbb{R})$ of the tangent bundle $(TK, \re \varpi)$ can be reduced to $\Sp (2, \mathbb{C})$. Here $\Sp (4, \mathbb{R})$ and $\Sp (2, \mathbb{C})$ are the real and complex symplectic groups respectively.
Therefore, the conclusion follows from the fact that $\Sp (2, \mathbb{C})$ is simply connected.
\end{proof}

Next, we will prove some topological properties of Kummer surfaces.

\begin{prop}{\cite[(8.6) in Page 257]{Barth_Peters_Van de Ven}}\label{proposition_kummer_1-connected}
Every Kummer surface $K$ is simply connected.
\end{prop}

The ring structure of $H^{*}(K; \mathbb{R})$ is well known (see e.g. \cite[Page 241]{Barth_Peters_Van de Ven}). However, most descriptions of $H^{*}(K; \mathbb{R})$ in the literature are rather abstract and algebraic, which is not enough for our purposes. We shall describe it more concretely and geometrically.
\begin{lemma}\label{lemma_kummer_odd_homology}
$H^{2j+1}(K; \mathbb{R}) = 0$ for all $j$.
\end{lemma}
\begin{proof}
It follows immediately from Proposition \ref{proposition_kummer_1-connected} and the Poincar\'e duality.
\end{proof}

\begin{lemma}\label{lemma_1_k_homology}
There exists a decomposition of the cohomology group
\begin{equation}\label{lemma_1_k_homology_1}
H^{2}(K; \mathbb{R}) = \bigoplus_{1\leq j\leq 16} \mathbb{R} \langle \PD (Y_{j}) \rangle \oplus H^{2} \left( K, \bigsqcup_{1\leq j\leq 16} Y_{j}; \mathbb{R} \right)
\end{equation}
where each $Y_{j}$ is the exceptional divisor defined in Lemma \ref{lemma_kummer}, $\PD (Y_{j}) \in H^{2}(K; \mathbb{Z})$ is the Poincar\'{e} dual of $Y_{j}$ and $\mathbb{R} \langle \PD (Y_{j}) \rangle$ is the real vector space generated by $\PD (Y_{j})$. All vector spaces on the right hand side are naturally subspaces of $H^{2}(K; \mathbb{R})$.
\end{lemma}
\begin{proof}
Given any $Y_j$, we can choose a small closed tubular neighborhood $\mathcal{N}_{j}$ of $Y_{j}$. These $\mathcal{N}_{j}$ are pairwise disjoint and they are homeomorphic to a tubular neighborhood (or a disk bundle) of $\mathbb{CP}^{1}$ in $\mathcal{H}^{2}$. Thus each $\partial \mathcal{N}_{j}$ is homeomorphic to $\mathbb{RP}^{3}$. Let $W = K \setminus \bigsqcup_{1\leq j\leq 16} \Int \mathcal{N}_{j}$, where $\Int \mathcal{N}_{j}$ is the interior of $\mathcal{N}_{j}$. Using the Mayer-Vietoris sequence, we have
\[
H^{2}(K; \mathbb{R}) =\left( \bigoplus_{1\leq j\leq 16} H^{2}(\mathcal{N}_{j}, \partial \mathcal{N}_{j}; \mathbb{R}) \right) \oplus H^{2}(W, \partial W; \mathbb{R}),
\]
where the vector spaces on the right side can be considered as subspaces of $H^{2}(K; \mathbb{R})$ by excision.

Since $\mathcal{N}_{j}$ is a disk bundle associated with $\mathcal{H}^{2}$ over $Y_{j}$, we infer that the Thom class $\alpha_{j} \in H^{2}(\mathcal{N}_{j}, \partial \mathcal{N}_{j}; \mathbb{R})$ is a generator of $H^{2}(\mathcal{N}_{j}, \partial \mathcal{N}_{j}; \mathbb{R})$. Moreover, $\alpha_{j} = \PD (Y_{j})$, if we consider $\alpha_{j}$ as an element in $H^{2}(K; \mathbb{R})$ (see Problem 11-C in \cite[p.\ 135]{milnor_stasheff} and \cite[p.\ 67]{Bott_Tu}). Thus $H^{2}(\mathcal{N}_{j}, \partial \mathcal{N}_{j}; \mathbb{R}) = \mathbb{R} \langle \PD (Y_{j}) \rangle$ as subspaces of $H^2(K; \mathbb{R})$.

By excision, $H^{2}(W, \partial W; \mathbb{R}) = H^{2} \left( K, \bigsqcup_{1\leq j\leq 16} \mathcal{N}_{j}; \mathbb{R} \right)$. Since $Y_{j}$ is a deformation retract of $\mathcal{N}_{j}$, we infer that $H^{2} \left( K, \bigsqcup_{1\leq j\leq 16} \mathcal{N}_{j}; \mathbb{R} \right) = H^{2} \left( K, \bigsqcup_{1\leq j\leq 16} Y_{j}; \mathbb{R} \right)$. Therefore (\ref{lemma_1_k_homology_1}) follows.
\end{proof}

\begin{lemma}\label{lemma_2_k_homology}
The following induced maps on cohomology groups are all isomorphisms,
\[
H^{2} \left( K, \bigsqcup_{1\leq j\leq 16} Y_{j}; \mathbb{R} \right) \overset{P^{*}}{\longleftarrow} H^{2} \left( \bar{T}, \bigsqcup_{1\leq j\leq 16}\{ [w_{j}]\}; \mathbb{R} \right) \overset{Q^{*}}{\longrightarrow} H^{2}(T; \mathbb{R})
\]
where $Q: T \rightarrow \bar{T}$ is the quotient map and $P$ is the blowup map in Lemma \ref{lemma_kummer}.

\end{lemma}
\begin{proof}
Clearly, $\bigsqcup_{1\leq j\leq 16} Y_{j}$ is a neighborhood deformation retract of $K$, and $\bigsqcup_{1\leq j\leq 16} \{ [w_{j}] \}$ is also a neighborhood deformation retract of $\bar{T}$. According to Lemma \ref{lemma_kummer}, $P: K \setminus \bigsqcup_{1\leq j\leq 16} Y_{j} \rightarrow \bar{T}\setminus \bigsqcup_{1\leq j\leq 16} \{ [w_{j}] \}$ is a homeomorphism. So $P^{*}$ is an isomorphism.

Consider the following commutative diagram,
\[
\xymatrix{
  H^{2} \left( \bar{T}, \bigsqcup_{1\leq j\leq 16} \{ [w_{j}] \}; \mathbb{R} \right) \ar[d] \ar[r]^-{Q^{*}} & H^{2} \left( T, \bigsqcup_{1\leq j\leq 16} \{ w_{j} \}; \mathbb{R} \right) \ar[d] \\
  H^{2} \left( \bar{T}; \mathbb{R} \right) \ar[d] \ar[r]^-{Q^{*}} & H^{2} (T; \mathbb{R}) \ar[d] \\
  H^{2} \left( \bar{T}\setminus \bigsqcup_{1\leq j\leq 16} \{ [w_{j}] \}; \mathbb{R} \right) \ar[r]^-{\hat{Q}^{*}} & H^{2} \left( T \setminus \bigsqcup_{1\leq j\leq 16} \{ w_{j} \}; \mathbb{R} \right)   }
\]
where $\hat{Q}$ is the restriction of $Q$ to $T \setminus \cup_{j=1}^{16} \{ w_{j} \}$, and all vertical maps are induced by the inclusions of spaces.

It is easy to check that all vertical maps are isomorphisms. Since
$$\hat{Q}: T \setminus \bigsqcup_{1\leq j\leq16} \{ w_{j} \} \rightarrow \bar{T} \setminus \bigsqcup_{1\leq j\leq 16} \{ [w_{j}] \}$$ is a double covering, $\hat{Q}^{*}$ is injective and its image contains exactly the elements fixed by the $C_2$ action
of $H^{2} \left( T \setminus \bigsqcup_{1\leq j\leq 16} \{ w_{j} \}; \mathbb{R} \right)$. Since $H^{2} \left( T \setminus \bigsqcup_{1\leq j\leq 16} \{ w_{j} \}; \mathbb{R} \right) = H^{2} (T; \mathbb{R})$ and the whole space $H^{2} (T; \mathbb{R})$ is fixed by $C_{2}$, we infer that $\hat{Q}^{*}$ is surjective. Therefore,
$$Q^{*}: H^{2} \left( \bar{T}, \bigsqcup_{1\leq j\leq 16} \{ [w_{j}] \}; \mathbb{R} \right) \rightarrow H^{2}(T; \mathbb{R})$$
is an isomorphism.
\end{proof}

Suppose $A \in \SL (2, \mathbb{Z} \sqrt{-1})$. Then $\det (A) =1$, and $A$ preserves the lattice $\Lambda$, i.e. $A (\Lambda) = \Lambda$. Hence $A$ descends to a complex Lie group automorphism of $T$, which we denote by $A_T$.

The following lemma is obvious.

\begin{lemma}\label{lemma_linear_map}
The following diagram commutes.
\[
\xymatrix{
  \mathbb{C}^{2} \ar[d]_{\pi} \ar[r]^{A} & \mathbb{C}^{2} \ar[d]_{\pi} \\
  T \ar[r]^{A_T} & T  }
\]
Moreover, the automorphism $A_T: T\to T$ induces a permutation of the points $w_{j} \;(j=1, \dots, 16)$.
\end{lemma}

Since $A_T: T\to T$ commutes with the $C_{2}$ action, it descends further to an automorphism of complex orbifold $A_{\bar{T}}: \bar{T}\to \bar{T}$.
Thus we can blow up $A_{\bar{T}}:\bar{T}\to \bar{T}$ to get a holomorphic map $A_K: K \rightarrow K$.

\begin{prop}\label{proposition_biholomorphic}
(1). The following diagram commutes, where the map $P$ is the blowup map defined in Lemma \ref{lemma_kummer}.
\[
\xymatrix{
  K \ar[d]_{P} \ar[r]^{A_K} & K \ar[d]^{P} \\
  \bar{T} \ar[r]^{A_{\bar{T}}} & \bar{T} }
\]
Moreover, $A_K$ is biholomorphic.

(2). The map $A_K$ permutes the exceptional divisors $Y_{j}$ ($j=1, \dots, 16$).
\end{prop}
\begin{proof}
Part (1) follows immediately from the definition of the blowup map $P$, and part (2) follows from Lemmas \ref{lemma_kummer} and \ref{lemma_linear_map}.
\end{proof}

\begin{prop}\label{proposition_kummer_symplectic_action}
Let $\varpi$ be the holomorphic symplectic form on $K$ defined in Remark \ref{holomorphic_symplectic_form}. Then,
\[
A_K^{*} \varpi = \varpi.
\]
\end{prop}
\begin{proof}
Since $\varpi$ is a holomorphic form, it suffices to show the equality on $K\setminus \bigsqcup_{1\leq j\leq 16}Y_j$. By Lemma \ref{lemma_kummer}, Proposition \ref{proposition_biholomorphic} (1) and by the definition of $\varpi$, it further suffices to show the following equality on $T$,
\[
A_{T}^{*} (dz_{1} \wedge dz_{2}) = dz_{1} \wedge dz_{2}.
\]
Since $A_{T}^{*} (dz_{1} \wedge dz_{2}) = \det(A) dz_{1}\wedge dz_{2}$ and $\det (A) = 1$, the proposition follows.
\end{proof}

Recall that $A\in \SL(2, \mathbb{Z}\sqrt{-1})$. Considering $A$ as an action on $\mathbb{C}^{2} = \mathbb{R}^{4}$, we denote the corresponding element in $\GL(4, \mathbb{R})$ by $A_{\mathbb{R}}$. Using equation (\ref{lemma_1_k_homology_1}), we can give an explicit description of the action of $A_{K}^{*}$ on $H^2(K, \mathbb{R})$.
\begin{prop}\label{proposition_complex_action_homology}
The subspaces $\bigoplus_{1\leq j\leq 16} \mathbb{R} \langle \PD (Y_{j})$ and $H^{2} \left( K, \bigsqcup_{1\leq j\leq 16} Y_{j}; \mathbb{R} \right)$ in (\ref{lemma_1_k_homology_1}) are $A_{K}^{*}$-invariant subspaces of $H^2(K; \mathbb{R})$. Moreover, suppose $A_T(w_{j}) = w_{k}$ as in Lemma \ref{lemma_linear_map}, then $A_K^{*} (\PD (Y_{k})) = \PD (Y_{j})$. The action of $A_K^{*}$ on $H^{2} \left( K, \bigsqcup_{1\leq j\leq 16} Y_{j}; \mathbb{R} \right)$ is isomorphic to the induced action of $A_{\mathbb{R}}$ on $\wedge^{2} (4, \mathbb{R})$, where $\wedge^{2} (4, \mathbb{R})$ is the space of the skew-symmetric bilinear forms on $\mathbb{R}^{4}$.
\end{prop}
\begin{proof}
By Proposition \ref{proposition_biholomorphic}, we know that $A_K (Y_{j}) = Y_{k}$ if $A_{T} (w_{j}) = w_{k}$. Obviously, $A_K$ preserves the orientations of all $Y_{j}$ and $K$. Thus, $A_K^{*} (\PD (Y_{k})) = \PD (Y_{j})$, and hence $\bigoplus_{1\leq j\leq 16} \mathbb{R} \langle \PD (Y_{j}) \rangle$ is $A_{K}^{*}$ invariant.

Since $A_K$ maps $\left( K, \bigsqcup_{1\leq j\leq 16} Y_{j} \right)$ to itself, $H^{2} \left( K, \bigsqcup_{1\leq j\leq 16} Y_{j}; \mathbb{R} \right)$ is also $A_K^{*}$ invariant. We also have the following commutative diagram,
\[
\xymatrix{
  H^{2} \left( K, \bigsqcup_{1\leq j\leq 16} Y_{j}; \mathbb{R} \right) \ar[d]_{A_K^{*}} & \ar[l]_-{P^{*}}  H^{2} \left( \bar{T}, \bigcup_{1\leq j\leq16} \{ [w_{j}] \}; \mathbb{R} \right) \ar[d]_{A_{\bar{T}}^{*}} \ar[r]^-{Q^{*}} & H^{2}(T; \mathbb{R}) \ar[d]^{A_T^{*}} \\
  H^{2} \left( K, \bigsqcup_{1\leq j\leq 16} Y_{j}; \mathbb{R} \right) & \ar[l]_-{P^{*}}  H^{2} \left( \bar{T}, \bigcup_{1\leq j\leq 16} \{ [w_{j}] \}; \mathbb{R} \right) \ar[r]^-{Q^{*}} & H^{2}(T; \mathbb{R}).   }
\]
By Lemma \ref{lemma_2_k_homology}, all horizontal maps are isomorphisms. Therefore, the action of $A_K^{*}$ on $H^{2} \left( K, \bigsqcup_{1\leq j\leq 16} Y_{j}; \mathbb{R} \right)$ is isomorphic to that of $A_T^{*}$ on $H^{2}(T; \mathbb{R})$.

Obviously, the action of $A_T^{*}$ on $H^{2}(T; \mathbb{R})$ is isomorphic to the action of $A_{\mathbb{R}}$ on $\wedge^{2} (4, \mathbb{R})$, which completes the proof.
\end{proof}

\begin{lemma}\label{lemma_gA_eigenvalue}
If as an element in $\SL(2, \mathbb{Z}\sqrt{-1})$, the matrix $A$ has eigenvalues $\lambda$ and $\lambda^{-1}$ (here we allow $\lambda=\lambda^{-1}$), then the set of eigenvalues of $A_K^{*}$ on $H^{2}(K; \mathbb{C})$ consists of $|\lambda|^{2}$, $|\lambda|^{-2}$, and some complex numbers with absolute value $1$.
\end{lemma}
\begin{proof}
By (\ref{lemma_1_k_homology_1}) and Proposition \ref{proposition_complex_action_homology}, we know that
\[
H^{2}(K; \mathbb{C}) = \bigoplus_{1\leq j\leq 16} \mathbb{C} \langle \PD (Y_{j}) \rangle \oplus H^{2} \left( K, \bigsqcup_{1\leq j\leq 16} Y_{j}; \mathbb{C} \right),
\]
where $V_{1} = \bigoplus_{1\leq j\leq 16} \mathbb{C} \langle \PD (Y_{j}) \rangle$ and $V_{2} = H^{2} \left( K, \bigsqcup_{1\leq j\leq 16} Y_{j}; \mathbb{C} \right)$ are $A_K^{*}$ invariant subspaces. Moreover, $A_K^{*}$ permutes $\PD (Y_{j})$. So $A_K^{*}|_{V_{1}}$ is unitary and hence its eigenvalues all have absolute value $1$.

The action of $A_K^{*}$ on $V_{2}$ is isomorphic to the action of $A_{\mathbb{R}}$ on $\wedge^{2} (4, \mathbb{C})\cong \wedge^2(4, \mathbb{R})\otimes_{\mathbb{R}}\mathbb{C}$.
Since $A_{\mathbb{R}} \in \GL (4, \mathbb{R})$ is defined from $A \in \GL (2, \mathbb{C})$, it has two invariant subspaces $W$ and $\bar{W}$ as an action on $\mathbb{C}^4$, such that $\mathbb{C}^{4} = W \oplus \bar{W}$, $A_{\mathbb{R}}|_{W} \cong A$ and $A_{\mathbb{R}}|_{\bar{W}} \cong \bar{A}$. Here $\bar{W}$ is the complex conjugate of $W$ and $\bar{A} \in \GL(2, \mathbb{C})$ is the complex conjugate of $A$. Thus the action of $A_{\mathbb{R}}$ on $\mathbb{C}^{4}$ exactly has eigenvalues $\lambda$, $\lambda^{-1}$, $\bar{\lambda}$ and $\bar{\lambda}^{-1}$. By Proposition \ref{proposition_complex_action_homology}, $A_{K}^{*}|_{V_{2}}$ exactly has eigenvalues $|\lambda|^{2}$, $|\lambda|^{-2}$, $\lambda \bar{\lambda}^{-1}$, $\bar{\lambda} \lambda^{-1}$, and $1$. Thus, the lemma follows.
\end{proof}

\section{The Main Theorem}\label{section_main}
In this section, we will construct our examples of nonk\"{a}hler manifolds and prove the main Theorem \ref{theorem_complex-symplectic}. We will assume Theorem \ref{theorem_monodromy} and Proposition \ref{proposition_formal}, and postpone their proofs to later sections.

Throughout this section $A$ is a matrix in $\SL (2, \mathbb{Z}[\sqrt{-1}])$. There is no additional restriction on $A$ unless we state it explicitly.

Endow $S^{1} \times \mathbb{R}^{1}$ with the standard complex structure and the standard symplectic structure, such that $S^{1} \times [0,1]$ has the symplectic area $1$. By the complex structure and the symplectic structure (as in Proposition \ref{proposition_kummer_SCY}) of $K$, the manifold $S^{1} \times \mathbb{R}^{1} \times K$ has its product complex and symplectic structures.

For any $A_{K}$ in Proposition \ref{proposition_biholomorphic}, one can define a $\mathbb{Z}$ action on $S^{1} \times \mathbb{R}^{1} \times K$ by
\begin{eqnarray}\label{complex_action}
\mathbb{Z} \times S^{1} \times \mathbb{R}^{1} \times K  &  \rightarrow  &  S^{1} \times \mathbb{R}^{1} \times K \nonumber \\
(n, a, b, c)   &  \rightarrow  &  (a, b+n, A_{K}^{n}(c))
\end{eqnarray}

\begin{prop}\label{proposition_M(A)}
The quotient space
\begin{equation}\label{proposition_M(A)_1}
M(A) = (S^{1} \times \mathbb{R}^{1} \times K) / \mathbb{Z}
\end{equation}
is a compact smooth manifold of dimension $6$. It has a unique complex (resp. symplectic) structure such that the quotient map $S^{1} \times \mathbb{R}^{1} \times K \rightarrow M(A)$ is holomorphic (resp. symplectic). The above complex structure and symplectic structure of $M(A)$ are compatible with its smooth structure.
\end{prop}
\begin{proof}
Since the $\mathbb{Z}$ action is smooth, proper and without fixed points, $M(A)$ is a smooth manifold. Notice that $M(A)$ is a fiber bundle over $S^1\times S^1$ with compact fiber $K$. Hence it is compact and has dimension $6$. Furthermore, the $\mathbb{Z}$ action is both holomorphic (Proposition \ref{proposition_biholomorphic}) and symplectic (Proposition \ref{proposition_kummer_symplectic_action}). Thus $M(A)$ has the desired complex and symplectic structures.
\end{proof}

\begin{rmk}
In fact, $M(A) = M(-A)$. This is because $A_{\bar{T}}= (-A)_{\bar{T}}$, and hence $A_{K}=  (-A)_{K}$.
\end{rmk}

Proposition \ref{proposition_M(A)} defines the manifold $M(A)$ in Theorem \ref{theorem_complex-symplectic}. Now, we need to verify that $M(A)$ satisfies all the properties in Theorem \ref{theorem_complex-symplectic} when $|\tr(A)|>2$.

In this section, we write $M_{C}(A)$ (or $M_{S}(A)$) when we want to emphasize the complex (or symplectic) structure on $M(A)$. Otherwise, we simply write $M(A)$.

\begin{prop}\label{proposition_holomorphic_CY}
The canonical line bundle $\mathcal{K}_{M_{C}(A)} = \bigwedge^{3} \left( T^{*} M_{C}(A) \right)$ is holomorphically trivial.
\end{prop}
\begin{proof}
It suffices to construct a holomorphic frame of $\mathcal{K}_{M_{C}(A)}$. Since $\mathcal{K}_{K}$ is holomorphically trivial, there exist holomorphic frames of $\mathcal{K}_{K}$, for example, the holomorphic symplectic form $\varpi$ in Proposition \ref{proposition_kummer_symplectic_action}. Furthermore, since $\mathbb{C}^{1}$ is a holomorphic covering of $S^{1} \times \mathbb{R}^{1}$, obviously, $dz$ is a holomorphic frame of $\mathcal{K}_{S^{1} \times \mathbb{R}^{1}}$, where $z$ is the standard coordinate of $\mathbb{C}^{1}$. Thus $dz \times \varpi$ is a holomorphic frame of $\mathcal{K}_{S^{1} \times \mathbb{R}^{1} \times K}$. By Proposition \ref{proposition_kummer_symplectic_action}, $dz \times \varpi$ is invariant under the $\mathbb{Z}$ action in (\ref{complex_action}). By (\ref{proposition_M(A)_1}), $dz \times \varpi$ descends to a desired holomorphic frame of $\mathcal{K}_{M_{C}(A)}$.
\end{proof}

\begin{prop}\label{proposition_symplectic_CY}
The symplectic manifold $M_{S}(A)$ is integrally Calabi-Yau. In other words, $c_{1} = 0$, where $c_{1} \in H^{2}(M_{S}(A); \mathbb{Z})$ is the first Chern class of $M_{S}(A)$ with respect to its symplectic structure.
\end{prop}
\begin{proof}
The tangent bundle of $S^{1} \times \mathbb{R}^{1} \times K$ splits symplectically into two subbundles
\[
T (S^{1} \times \mathbb{R}^{1} \times K) = \tilde{E}_{H} \oplus \tilde{E}_{V},
\]
where $\tilde{E}_{H} = T (S^{1} \times \mathbb{R}^{1}) \times K$ and $\tilde{E}_{V} = S^{1} \times \mathbb{R}^{1} \times TK$. Since the $\mathbb{Z}$ action in (\ref{complex_action}) preserves the above splitting, it descends to the following splitting
\[
T M_{S}(A) = E_{H} \oplus E_{V},
\]
where $E_{H}$ comes from $\tilde{E}_{H}$, and $E_{V}$ comes from $\tilde{E}_{V}$.

Let $\varpi$ be the holomorphic symplectic form in Proposition \ref{proposition_kummer_symplectic_action}. Clearly, the restriction of $\re (1 \times \varpi)$ on $\tilde{E}_{V}$ yields the symplectic structure of $\tilde{E}_{V}$. Therefore, the structure group of $\tilde{E}_{V}$ can be reduced from $\Sp (4, \mathbb{R})$ to $\Sp (2, \mathbb{C})$. Since the $\mathbb{Z}$ action preserves $1 \times \varpi$, this form descends to a form on $M_{S}(A)$. Thus the structure group of $E_{V}$ can be also reduced to $\Sp (2, \mathbb{C})$, which implies the first Chern class of $E_{V}$ vanishes.

On the other hand, $E_{H} = p^{*} T (S^{1} \times S^{1})$ as symplectic bundles, where $p: M_{S}(A) \rightarrow S^{1} \times S^{1}$ is the bundle projection. Since $S^{1} \times S^{1}$ is Calabi-Yau, the first Chern class of $E_{H}$ also vanishes.

In summary, the first Chern class of $M_{S}(A)$ vanishes.
\end{proof}

\begin{prop}\label{proposition_complex_Betti}
The fundamental group $\pi_{1} (M(A)) = \mathbb{Z} \oplus \mathbb{Z}$, and the odd-degree Betti numbers of $M(A)$ are even.
\end{prop}
\begin{proof}
Since $M(A)$ is a fiber bundle with base $S^{1} \times S^{1}$ and fiber $K$, we have the homotopy long exact sequence of fibrations
\[
\pi_{1}(K) \rightarrow \pi_{1} (M(A)) \rightarrow \pi_{1} (S^{1} \times S^{1}) \rightarrow \pi_{0} (K).
\]
According to Proposition \ref{proposition_kummer_1-connected}, $\pi_{1}(K) = 0$. Obviously $\pi_{0}(K)$ is a one-point set. Therefore, $\pi_{1} (M(A)) \cong \pi_{1} (S^{1} \times S^{1}) \cong \mathbb{Z} \oplus \mathbb{Z}$, and hence the first Betti number of $M(A)$ is $2$. Since $\dim (M(A)) = 6$, by Poincar\'{e} duality, the 5th Betti number is also $2$. It also follows from Poincar\'{e} duality that the cup product pairing
\[
H^{3} (M(A); \mathbb{R}) \otimes H^{3} (M(A); \mathbb{R}) \overset{\cup}{\longrightarrow} H^{6} (M(A); \mathbb{R}) \cong \mathbb{R}
\]
is non-degenerate. This pairing is skew-symmetric. Thus, the dimension of $H^{3} (M(A); \mathbb{R})$ is even.
\end{proof}

Suppose $F$ is a topological space and $\eta: F \rightarrow F$ is a homeomorphism. Consider the $\mathbb{Z}$ action on the topological space $\mathbb{R}^{1} \times F$ defined by
\begin{eqnarray}\label{manifold_N}
\mathbb{Z} \times \mathbb{R}^{1} \times F & \longrightarrow & \mathbb{R}^{1} \times F \\
(n, t, a) & \longrightarrow & (n+t, \eta^{n}(a)) \nonumber
\end{eqnarray}
The quotient space $N = (\mathbb{R}^{1} \times F)/ \mathbb{Z}$ is a fiber bundle over $S^{1}$ with fiber $F$. Denote by $p_{2}: N \rightarrow S^{1}$ the projection. Let $i: F \rightarrow N$ be the natural inclusion map to the fiber over $0 \in \mathbb{R}^{1} / \mathbb{Z} = S^{1}$.

\begin{lemma}\label{lemma_wang}
We have the following long exact sequence
\[
\longrightarrow H^{j}(N; \mathbb{R}) \overset{i^{*}}{\longrightarrow} H^{j}(F; \mathbb{R}) \overset{\eta^{*} - \Id}{\longrightarrow} H^{j}(F; \mathbb{R}) \overset{\delta}{\longrightarrow} H^{j+1}(N; \mathbb{R}) \longrightarrow,
\]
where $\Id$ is the identity map. If $\beta \in H^{j}(F; \mathbb{R})$ such that $i^{*} \alpha = \beta$ with $\alpha \in H^{j}(N; \mathbb{R})$, then $\delta (\beta) = p_{2}^{*}s \cup \alpha$. Here $s$ is the generator in $H^{1}(S^{1}; \mathbb{Z})$, which represents the positive orientation of $S^{1}$.
\end{lemma}

Lemma \ref{lemma_wang} is a special case of the Wang sequence which studies fibrations over spheres $S^{n}$ (see (1.9) in \cite[p.\ 319]{G.Whitehead} and the proposition 6.4.8 in \cite[p.\ 212]{dimca}). In the literature, the Wang sequence is usually proved in the case of $n>1$. Nevertheless, the proof of the case $n=1$ using the Mayer-Vietoris sequence is similar. So we skip the proof here.

Taking $F = K$ and $\eta = A_{K}$ in (\ref{manifold_N}), we get a fiber bundle $N = (\mathbb{R}^{1} \times K) / \mathbb{Z}$ over $S^{1}$ with fiber $K$. By definition, $M(A) = S^{1} \times N$.

\begin{lemma}\label{lemma_even_complex_homology}
The homomorphism $i^{*}: H^{2j}(N; \mathbb{R}) \rightarrow H^{2j}(K; \mathbb{R})$ is an isomorphism when $j=0$, $2$. When $j=1$, $i^{*}$ is injective, and its image $\im i^{*} = V$, where $V$ is the set of fixed elements of $A_{K}^{*}$ in $H^{2}(K; \mathbb{R})$.
\end{lemma}
\begin{proof}
By Lemma \ref{lemma_wang}, we have the following Wang sequence
\[
H^{2j-1}(K; \mathbb{R}) \overset{\delta}{\longrightarrow} H^{2j}(N; \mathbb{R}) \overset{i^{*}}{\longrightarrow} H^{2j}(K; \mathbb{R}) \overset{A_{K}^{*} - \Id}{\longrightarrow} H^{2j}(K; \mathbb{R}).
\]
By Lemma \ref{lemma_kummer_odd_homology}, $H^{2j-1}(K; \mathbb{R}) = 0$. Therefore, $i^{*}$ is injective and its image is the subspace of $H^{2j}(K; \mathbb{R})$ fixed by $A_{K}^{*}$.

For $j=0$, $2$, the whole space $H^{2j}(K; \mathbb{R})$ is fixed by $A_{K}^{*}$. Thus $i^{*}$ is an isomorphism for $j=0$, $2$.
\end{proof}

\begin{lemma}\label{lemma_diagoanl_complex}
If $A$ is diagonalizable, then $A_{K}^{*}$ is diagonalizable on $H^{2}(K; \mathbb{C})$.
\end{lemma}
\begin{proof}
By Proposition \ref{proposition_complex_action_homology}, it suffices to show that the restrictions of $A_{K}^{*}$ to the direct summands $\bigoplus_{1\leq j\leq 16} \mathbb{C} \langle \PD (Y_{j}) \rangle$ and $H^{2} \left( K, \bigsqcup_{1\leq j\leq 16} Y_{j}; \mathbb{C} \right)$ are both diagonalizable. By Proposition \ref{proposition_complex_action_homology}, $A_{K}^{*}$ permutes these $\PD (Y_{j})$. Thus the restriction of $A_{K}^{*}$ on $\bigoplus_{1\leq j\leq 16} \mathbb{R} \langle \PD (Y_{j})$ is unitary and hence diagonalizable. The restriction of $A_{K}^{*}$ on $H^{2} \left( K, \bigsqcup_{1\leq j\leq 16} Y_{j}; \mathbb{C} \right)$ is isomorphic to the induced action of $A_{\mathbb{R}}$ on $\wedge^{2}(4, \mathbb{C})$, and hence diagonalizable too.
\end{proof}

\begin{lemma}\label{lemma_odd_complex_homology}
Let $s \in H^{1}(S^{1}; \mathbb{Z})$ be the generator corresponding to the positive orientation. Denote by $s_{2} = p_{2}^{*} s \in H^{1} (N; \mathbb{Z})$.

When $j=0 \text{ or } 2$, the map
\begin{equation}\label{lemma_odd_complex_homology_1}
\xymatrix{
  H^{2j}(N; \mathbb{R}) \ar[r]^-{s_{2} \cup} & H^{2j+1}(N; \mathbb{R})   }
\end{equation}
is an isomorphism. Here $s_{2} \cup$ is the map defined by taking cup product with $s_{2}$.

Assume additionally $A$ is diagonalizable. Then (\ref{lemma_odd_complex_homology_1}) is also an isomorphism for $j=1$.
\end{lemma}
\begin{proof}
By Lemma \ref{lemma_wang}, we have the following Wang sequence
\[
H^{2j}(K; \mathbb{R}) \overset{A_{K}^{*} - \Id}{\longrightarrow} H^{2j}(K; \mathbb{R}) \overset{\delta}{\longrightarrow} H^{2j+1}(N; \mathbb{R}) \overset{i^{*}}{\longrightarrow} H^{2j+1}(K; \mathbb{R}).
\]
By Lemma \ref{lemma_kummer_odd_homology}, $H^{2j+1}(K; \mathbb{R})=0$, and hence $\delta$ is surjective.

Clearly, $A_{K}^{*} - \Id =0$ on $H^{2j}(K; \mathbb{R})$, when $j=0 \text{ or } 2$. If additionally $A$ is diagonalizable, by Lemma \ref{lemma_diagoanl_complex}, then $A_{K}^{*} - \Id$ is diagonalizable on $H^{2}(K; \mathbb{C})$. This assumption further implies that, on $H^{2}(K; \mathbb{R})$, the kernel of $A_{K}^{*} - \Id$ is a complement of the image of $A_{K}^{*} - \Id$.

Thus, when $j=0 \text{ or } 2$, the map
\[
\delta: \ker [(A_{K}^{*} - \Id)|_{H^{2j}(K; \mathbb{R})}] \longrightarrow H^{2j+1}(N; \mathbb{R})
\]
is an isomorphism. When $j=1$ and when $A$ is diagonalizable, $\delta$ is also an isomorphism.

On the other hand, by Lemma \ref{lemma_even_complex_homology},
\[
i^{*}: H^{2j}(N; \mathbb{R}) \longrightarrow \ker [(A_{K}^{*} - \Id)|_{H^{2j}(K; \mathbb{R})}]
\]
is an isomorphism. Combining the above two isomorphisms,
\[
\delta \circ i^{*}: H^{2j}(N; \mathbb{R}) \longrightarrow H^{2j+1}(N; \mathbb{R})
\]
is also an isomorphism. By Lemma \ref{lemma_wang}, $\delta \circ i^{*} (\alpha) = s_{2} \cup \alpha$, which completes the proof.
\end{proof}

By Lemmas \ref{lemma_even_complex_homology}, \ref{lemma_odd_complex_homology} and the K\"{u}nneth formula, we obtain the following description of $H^{*}(M(A); \mathbb{R})$.

\begin{lemma}\label{lemma_complex_homology}
Let $s_{1} \in H^{1}(S^{1}; \mathbb{Z})$ and $[K]^{*}\in H^{4}(K; \mathbb{Z})$ be the generators corresponding to the positive orientations.

For $M(A) = S^{1} \times N$, we have
\begin{eqnarray*}
H^{0}(M(A); \mathbb{R}) & = & \mathbb{R}, \\
H^{1}(M(A); \mathbb{R}) & = & \mathbb{R} \langle s_{1} \times 1 \rangle \oplus \mathbb{R} \langle 1 \times s_{2} \rangle, \\
H^{2}(M(A); \mathbb{R}) & = & \mathbb{R} \langle s_{1} \times s_{2} \rangle \oplus (1 \times V), \\
H^{5}(M(A); \mathbb{R}) & = & \mathbb{R} \langle s_{1} \times [K]^{*} \rangle \oplus \mathbb{R} \langle 1 \times (s_{2} \cup [K]^{*}) \rangle, \\
H^{6}(M(A); \mathbb{R}) & = & \mathbb{R} \langle s_{1} \times (s_{2} \cup [K]^{*}) \rangle.
\end{eqnarray*}

Assume additionally that $A$ is diagonalizable. Then
\begin{eqnarray*}
H^{3}(M(A); \mathbb{R}) & = & (s_{1} \times V) \oplus [1 \times (s_{2} \cup V)], \\
H^{4}(M(A); \mathbb{R}) & = & [s_{1} \times (s_{2} \cup V)] \oplus \mathbb{R} \langle 1 \times [K]^{*} \rangle.
\end{eqnarray*}

Here $s_{2} \in H^{1} (N; \mathbb{Z})$ is the element defined in Lemma \ref{lemma_odd_complex_homology}, and $V$ is the subspace of $H^{2}(K; \mathbb{R})$ defined in Lemma \ref{lemma_even_complex_homology}. We identify $H^{2}(N; \mathbb{R})$ with $V$ via the isomorphism $i^{*}: H^{2}(N; \mathbb{R}) \rightarrow V$, and we also consider $[K]^{*}$ as an element in $H^{4}(N; \mathbb{R})$ via the isomorphism $i^{*}: H^{4}(N; \mathbb{R}) \rightarrow H^{4}(K; \mathbb{R})$.
\end{lemma}

Now we study the hard Lefschetz property of $M_{S}(A)$.
\begin{lemma}\label{lemma_symplectic_class}
Let $[\omega]\in H^{2}(M_{S}(A); \mathbb{R})$ be the cohomology class represented by the symplectic form $\omega$ of $M_S(A)$. Then
\[
[\omega]= s_{1} \times s_{2} + 1 \times \theta
\]
where $\theta \in V$ and $\theta^{2} = d [K]^{*}$ with $d>0$. Here $s_{1}$, $s_{2}$, $[K]^{*}$ and $V$ are as in Lemma \ref{lemma_complex_homology}.

\end{lemma}
\begin{proof}
Putting together the standard symplectic form $\omega_{1}$ on $S^{1} \times \mathbb{R}^{1}$ and the $\re \varpi$ on $K$, we obtain a product symplectic form on $S^{1} \times \mathbb{R}^{1} \times K$. This symplectic form descends to the symplectic form $\omega$ on $M_S(A)$. The form $\omega_{1}$ also descends to a symplectic form on $S^{1} \times S^{1}$, which we also denote by $\omega_{1}$. Then, $\int_{S^{1} \times S^{1}} \omega_{1} = 1$.

Clearly, $1 \times \re \varpi$ is a closed form on $\mathbb{R}^{1} \times K$. By Proposition \ref{proposition_kummer_symplectic_action}, $1 \times \re \varpi$ descends to a closed form $\omega_{2}$ on $N$. Then
\[
\omega = p^{*} \omega_{1} + 1 \times \omega_{2},
\]
where $p: M_{S}(A) \rightarrow S^{1} \times S^{1}$ is the projection. This implies $[\omega] = s_{1} \times s_{2} + 1 \times \theta$, where $\theta = [\omega_{2}] \in H^{2}(N; \mathbb{R})$, and $H^{2}(N; \mathbb{R})$ is identified with the $V$ in Lemma \ref{lemma_complex_homology}.

Since $(\re \varpi)^{2}$ is a volume form on $K$, we have $[\re \varpi]^{2} = d [K]^{*} \in H^{4}(K; \mathbb{R})$ with $d>0$. Furthermore, the restriction of $\omega_{2}$ on a fiber $K$ is equal to $\re \varpi$. Therefore, $\theta^{2} = d [K]^{*}$, where $H^{4}(N; \mathbb{R})$ is identified with $H^{4}(K; \mathbb{R})$ by Lemma \ref{lemma_complex_homology}.
\end{proof}

\begin{prop}\label{proposition_symplectic_Lefschetz}
Assume $A$ is diagonalizable. Then, for any $j$,
\[
L_{[\omega]}^{j}: H^{3-j}(M_{S}(A); \mathbb{R}) \overset{[\omega]^{j} \cup}{\longrightarrow} H^{3+j}(M_{S}(A); \mathbb{R})
\]
is an isomorphism, where $[\omega] \in H^{2}(M_{S}(A); \mathbb{R})$ is the cohomology class represented by the symplectic form $\omega$ of $M_{S}(A)$.
\end{prop}
\begin{proof}
Since $A$ is diagonalizable, we can use all the conclusions of Lemmas \ref{lemma_odd_complex_homology} and \ref{lemma_complex_homology}.

The statement is trivial if $j\neq 1$, $2$ or $3$. Moreover, it suffices to show that $L_{[\omega]}^{j}$ is injective, because $\dim (H^{3-j}) = \dim (H^{3+j})$.  By Lemma \ref{lemma_symplectic_class}, we have $[\omega]= s_{1} \times s_{2} + 1 \times \theta$. Note that $s_{1}^{2} = 0$, $s_{2}^{2} = 0$, $\theta^{3} = 0$, and $\theta^{2} = d[K]^{*}$ with $d>0$.

When $j=3$, since $\omega^{3}$ is a volume form of $M_{S}(A)$, $L_{[\omega]}^{3}$ is an isomorphism.

When $j=2$,
\[
[\omega]^{2} = 2 s_{1} \times (s_{2} \cup \theta) + 1 \times \theta^{2} = 2 s_{1} \times (s_{2} \cup \theta) + d \times [K]^{*}.
\]
By Lemma \ref{lemma_complex_homology}, for any $\beta \in H^{1}(M_{S}(A); \mathbb{R})$, it is of the form $\beta = s_{1} \times a + b \times s_{2}$, with $a,b \in \mathbb{R}$. Then
\[
L_{[\omega]}^{2} (\beta) = ad s_{1} \times [K]^{*} + bd \times (s_{2} \cup [K]^{*}).
\]
By Lemma \ref{lemma_odd_complex_homology}, $s_{2} \cup [K]^{*} \neq 0$. By Lemma \ref{lemma_complex_homology}, $s_1\times [K]^{*}$ and $1\times (s_2\cup [K]^{*})$ are linearly independent in $H^5(K; \mathbb{R})$. Since $d \neq 0$, if $L_{[\omega]}^{2} (\beta) = 0$, then $a=b=0$ and $\beta = 0$. Therefore, $L_{[\omega]}^{2}$ is injective.

When $j=1$, note that: (i) $\theta \in V$; (ii) $\theta \cup V \subseteq H^{4}(K; \mathbb{R})$; (iii) $\theta^{2} \neq 0$; and (iv) $\dim \left( H^{4}(K; \mathbb{R}) \right) = 1$. We infer that $V = \mathbb{R} \langle \theta \rangle \oplus V_{2}$, where $V_{2} = \{ \gamma \in V \mid \theta \cup \gamma = 0 \}$. By Lemma \ref{lemma_complex_homology}, for any $\beta \in H^{2}(M_{S}(A); \mathbb{R})$, it is of the form $\beta = a s_{1} \times s_{2} + b \times \theta + 1 \times \gamma$, where $a,b \in \mathbb{R}$, and $\gamma \in V_{2}$. Since $\theta^2=d[K]^*$ and $\theta\cup\gamma=0$, we have
\begin{eqnarray*}
&   & L_{[\omega]}^{1} (\beta) \\
& = & (a+b) s_{1} \times (s_{2} \cup \theta) + s_{1} \times (s_{2} \cup \gamma) + b \times \theta^{2} + 1 \times (\theta \cup \gamma) \\
& = & (a+b) s_{1} \times (s_{2} \cup \theta) + s_{1} \times (s_{2} \cup \gamma) + bd \times [K]^{*}.
\end{eqnarray*}

Since $\theta \neq 0$ and by Lemma \ref{lemma_odd_complex_homology}, we have $s_{2} \cup \theta \neq 0$. By Lemma \ref{lemma_complex_homology}, $s_{1} \times (s_{2} \cup \theta)$ and $1\times [K]^{*}$ are linearly independent in $H^4(K;\mathbb{R})$. If $\gamma\neq 0$, then $s_{1} \times (s_{2} \cup \theta)$, $s_{1} \times (s_{2} \cup \gamma)$ and $1\times [K]^{*}$ are also linearly independent. Recall that $d\neq 0$. Suppose $L_{[\omega]}^{1} (\beta) = 0$. Then $a+b=0$, $b=0$ and $\gamma=0$, and hence $\beta =0$. Therefore, $L_{[\omega]}^{1}$ is injective.
\end{proof}

Now we are ready to prove Theorem \ref{theorem_complex-symplectic}.
\begin{proof}[Proof of Theorem \ref{theorem_complex-symplectic}]
Since $|\tr(A)|>2$, we know that $A$ has two distinct eigenvalues $\lambda$ and $\lambda^{-1}$ with $|\lambda|>1$. Therefore, $A$ is diagonalizable.

Take $M_{C} = M_{C}(A)$ and $M_{S} = M_{S}(A)$. Then the properties (1) - (7) follow from Propositions \ref{proposition_M(A)}, \ref{proposition_holomorphic_CY}, \ref{proposition_symplectic_CY}, \ref{proposition_complex_Betti}, \ref{proposition_symplectic_Lefschetz}, and \ref{proposition_formal}.

We prove (8) by contradiction. Suppose $M(A) \times Y$ is homotopy equivalent to a compact K\"{a}hler manifold $X$.

Since $H^{*}(M(A); \mathbb{C})$ is finite dimensional, we have the K\"{u}nneth formula
\[
H^{k}(X; \mathbb{C}) \cong \bigoplus_{i+j=k} H^{i}(M(A); \mathbb{C}) \otimes H^{j}(Y; \mathbb{C}).
\]
Since $H^{*} (X; \mathbb{C})$ is finite dimensional, so is $H^{*}(Y; \mathbb{C})$. As the argument can be applied on each path component of $Y$, we may assume that $Y$ is path-connected.

Recall that $M(A) = S^{1} \times N$ and $N$ is a fiber bundle over $S^{1}$ with fiber a Kummer surface $K$. Then $M(A) \times Y$ is a fiber bundle over $S^{1}$ with path-connected fiber $S^{1} \times K \times Y$. By the K\"{u}nneth formula, $H^{*}(S^{1} \times K \times Y; \mathbb{C})$ are finite dimensional. Hence we can apply Theorem \ref{theorem_monodromy} to $\phi: M(A) \times Y \rightarrow S^{1}$, where $\phi$ is the composition of the projection $M(A)\times Y\to M(A)$ and the bundle map $M(A) \rightarrow S^1$. By (\ref{complex_action}), there is a monodromy map
\[
\Id \times A_{K} \times \Id: S^{1} \times K \times Y \rightarrow S^{1} \times K \times Y
\]
which yields a monodromy action on $H^{*}(S^{1} \times K \times Y; \mathbb{C})$, where $\Id$ are the identity maps. By the K\"{u}nneth formula again, $H^{2}(S^{1} \times K \times Y; \mathbb{C})$ has an $(\Id \times A_{K} \times \Id)^{*}$ invariant subspace $H^{0}(S^{1}; \mathbb{C}) \otimes H^{2}(K; \mathbb{C}) \otimes H^{0}(Y; \mathbb{C})$, on which $(\Id \times A_{K} \times \Id)^{*} = 1 \otimes A_{K}^{*} \otimes 1$. Since $A$ has an eigenvalue $\lambda$, by Lemma \ref{lemma_gA_eigenvalue}, $(\Id \times A_{K} \times \Id)^{*}$ has a real eigenvalue $|\lambda|^{2}$ on $H^{2}(S^{1} \times K \times Y; \mathbb{C})$. Therefore, the monodromy action of the fiber bundle $\phi: M(A) \times Y \rightarrow S^{1}$ has an eigenvalue $|\lambda|^2$. Since $|\lambda|>1$, we have $|\lambda|^2>1$.

However, we have assumed that $M(A) \times Y$ is homotopy equivalent to a compact K\"{a}hler manifold. By Theorem \ref{theorem_monodromy}, the eigenvalues of this monodromy action have to be roots of unity. This is a contradiction to the fact that $|\lambda|^2>1$.
\end{proof}

\begin{ex}\label{example_complex-symplectic}
There is a large family of $A \in \SL(2, \mathbb{Z}[\sqrt{-1}])$ satisfying the assumption of Theorem \ref{theorem_complex-symplectic}. For instance, we can choose $A$ to be
\[
\begin{bmatrix}
1 + \sqrt{-1} & \sqrt{-1} \\
1 & 1
\end{bmatrix},
\qquad
\begin{bmatrix}
2 & -1 + 4 \sqrt{-1} \\
1 & 2\sqrt{-1}
\end{bmatrix}
\qquad \text{or} \qquad
\begin{bmatrix}
1 & n-2 \\
1 & n-1
\end{bmatrix}
\]
for any integer $n>2$.
\end{ex}

\section{Formality}\label{section_formality}
In this section, we prove the following proposition which is Theorem \ref{theorem_complex-symplectic} (7).

\begin{prop}\label{proposition_formal}
The real homotopy type of $M(A)$ in Theorem \ref{theorem_complex-symplectic} is formal.
\end{prop}

Let us first recall the definition of the real homotopy type of a manifold being formal.

A commutative differential graded algebra (CDGA) over a field $\mathbb{K}$ of characteristic $0$, denoted by $(\mathcal{A}, d)$, is a graded algebra $\mathcal{A} = \bigoplus_{k \geq 0} \mathcal{A}^{k}$ with a differential $d$ which satisfies the following conditions:
\begin{enumerate}[(1)]
\item for all $x \in \mathcal{A}^{k}$ and $y \in \mathcal{A}^{l}$, we have $xy = (-1)^{kl} yx$;

\item the differential $d: \mathcal{A} \rightarrow \mathcal{A}$ has degree $1$, i.e. for any $k$, we have $d (\mathcal{A}^{k}) \subseteq \mathcal{A}^{k+1}$, and $d^2=0$;

\item for all $x \in \mathcal{A}^{k}$ and $y \in \mathcal{A}$, we have $d(xy) = dx \cdot y + (-1)^{k} x \cdot dy$.
\end{enumerate}

If $x \in \mathcal{A}^{k}$, we say $x$ is a homogenous element with degree $k$, and denote by $\deg (x) =k$.

With differential $d$, a CDGA $(\mathcal{A}, d)$ is naturally a complex of $\mathbb{K}$-vector spaces. Hence we can define its cohomology groups $H^{i}(\mathcal{A}, d)$. Notice that with trivial differential, $(H^{*}(\mathcal{A}, d), 0)$ has a natural CDGA structure. A homomorphism of CDGA is a homomorphism of graded algebra which is also a map of complexes.

\begin{defn}\label{definition_minimal}
A CDGA $(\mathcal{M}, d)$ is called minimal if it satisfies the following conditions:
\begin{enumerate}[(1)]
\item it is a free commutative graded algebra;

\item there exists a collection of homogeneous free generators $\{ a_{\tau} \mid \tau \in \mathcal{I} \}$, for some well ordered index set $\mathcal{I}$, such that $\deg (a_{\mu}) \leq \deg (a_{\tau})$ if $\mu, \tau \in \mathcal{I}$ and $\mu < \tau$. Moreover, $d a_{\tau}$ is expressed in terms of finitely many $a_{\mu}$ with $\mu < \tau$.
\end{enumerate}
\end{defn}

Here a free commutative graded algebra is the tensor product of the polynomial algebra generated by its free generators of even degrees and the exterior algebra generated by its free generators of odd degrees. For more discussions about minimal CDGAs, see \cite[p.\ 14]{Halperin} and \cite[p.\ 152]{Fernandez_Mumoz}.

\begin{defn}\label{definition_minimal_model}
Suppose $\mathcal{A}$ is a CDGA and $\mathcal{M}$ is a minimal CDGA. If there is a homomorphism of CDGA $\varphi: \mathcal{M} \rightarrow \mathcal{A}$ which induces isomorphisms $\varphi^{*}: H^{*}(\mathcal{M}) \rightarrow H^{*}(\mathcal{A})$ of cohomology groups, then $\mathcal{M}$ is a minimal model of $\mathcal{A}$.
\end{defn}

A CDGA $\mathcal{A}$ over $\mathbb{K}$ is called connected if $H^{0}(\mathcal{A}) = \mathbb{K}$. (Here, $\mathbb{K} \subseteq \mathcal{A}^{0}$ as a subalgebra and the differential $d$ vanishes on $\mathbb{K}$.) If $\mathcal{A}$ is connected, then there exists a unique minimal model $\mathcal{M}$ of $\mathcal{A}$ up to isomorphisms (see \cite[chapter 6]{Halperin}).

\begin{defn}\cite[p.\ 260]{D_G_M_S}\label{definition_formal}
Suppose $\mathcal{M}$ is a minimal model of $\mathcal{A}$. If there exists a homomorphism of CDGA $\psi: (\mathcal{M}, d) \rightarrow (H^{*}(\mathcal{A}), d=0)$ which induces the identity isomorphism of $H^{*} (\mathcal{A})$, then $\mathcal{A}$ is called formal.
\end{defn}

Suppose $M$ is a smooth manifold. Denote by $\Omega M$ the de Rham complex of differential forms on $M$. Then $\Omega M$ is a CDGA over $\mathbb{R}$ with cohomology $H^{*}(\Omega M) \cong H^{*} (M; \mathbb{R})$.

\begin{defn}\label{definition_formal_manifold}
The real homotopy type of a smooth manifold $M$ is formal, or briefly $M$ is formal, if $\Omega M$ is formal.
\end{defn}

We have explained the content of Proposition \ref{proposition_formal}. To prove this proposition, we need the definition of $s$-formal, as in \cite[definition 2.2]{Fernandez_Mumoz}. Suppose $\mathcal{M}$ has homogeneous free generators $a_{\tau}$ ($\tau \in \mathcal{I}$). Denote by $\mathcal{V}^{i}$ the linear space spanned by $\{ a_{\tau} \mid \tau \in \mathcal{I}, \deg (a_{\tau}) =i \}$, i.e. $\mathcal{V}^{i}$ is spanned by the homogeneous free generators of degree $i$. Clearly,  $\mathcal{V}^{i} \subseteq \mathcal{M}^{i}$. It is necessary to point out that, for a fixed $\mathcal{M}$, there are different choices of homogeneous free generators. The space $\mathcal{M}^{i}$ is independent of these choices. However, $\mathcal{V}^{i}$ does depend on the choices. In fact, this observation has been employed in the proof of Theorem \ref{theorem_FM-(n-1)-formal} below. Let $X$ be a subset of $\mathcal{M}$. Denote by $\bigwedge (X)$ the subalgebra over $\mathbb{K}$ generated by $X$ in $\mathcal{M}$.

\begin{defn}\label{definition_s-formal}
Suppose $(\mathcal{A}, d)$ has a minimal model $(\mathcal{M}, d)$. We say $\mathcal{A}$ is $s$-formal for some integer $s \geq 0$ or $s = + \infty$ if we can choose homogeneous free generators of $\mathcal{M}$ such that $\mathcal{V}^{i} = \mathcal{C}^{i} \oplus \mathcal{N}^{i}$, where $\mathcal{C}^{i}$ and $\mathcal{N}^{i}$ satisfy the following conditions:
\begin{enumerate}[(1)]
\item $d (\mathcal{C}^{i}) = 0$;

\item $d: \mathcal{N}^{i} \rightarrow \mathcal{M}$ is injective;

\item any cocycle in the ideal $I (\bigoplus_{i \leq s} \mathcal{N}^{i})$, generated by $\bigoplus_{i \leq s} \mathcal{N}^{i}$ in $\bigwedge \left( \bigoplus_{i \leq s} \mathcal{V}^{i} \right)$, is exact in $\mathcal{M}$.
\end{enumerate}
We say a smooth manifold $M$ is $s$-formal if its de Rham complex $\Omega M$ is $s$-formal.
\end{defn}

\begin{theorem}\cite[(4.1)]{D_G_M_S}\label{theorem_DGMS_formal}
A CDGA $\mathcal{A}$ is formal if and only if it is $+\infty$-formal.
\end{theorem}

For a manifold $M$ of dimension $n$, we have $H^{i}(\Omega M) = 0$ for $i>n$. Thus when $i>n$, a cocycle of degree $i$ in a minimal model of $\Omega M$ must be exact. Theorem \ref{theorem_DGMS_formal} immediately implies the following result.
\begin{theorem}\cite[lemma 2.10]{Fernandez_Mumoz}\label{theorem_FM_n-formal}
Suppose $M$ is an $n$-dimensional smooth manifold. Then $M$ is formal if and only if it is $n$-formal.
\end{theorem}

Furthermore, for a connected and orientable compact manifold, Fern\'{a}ndez and Mu\~{n}oz proved in \cite{Fernandez_Mumoz} the following powerful theorem which is important for our proof of Proposition \ref{proposition_formal}. The key idea of their proof is that one can improve Theorem \ref{theorem_FM_n-formal} by taking advantage of Poincar\'{e} duality.
\begin{theorem}\cite[theorem 3.1]{Fernandez_Mumoz}\label{theorem_FM-(n-1)-formal}
Suppose $M$ is a connected and orientable compact smooth manifold of dimension $2n$ or $2n-1$. Then $M$ is formal if and only if it is $(n-1)$-formal.
\end{theorem}

Recall that $M(A) = S^{1} \times N$. We shall apply Theorem \ref{theorem_FM-(n-1)-formal} to $N$ by constructing a minimal model of $\Omega N$. As mentioned above, every connected CDGA $\mathcal{A}$ has a minimal model $\mathcal{M}$. Let us briefly describe the construction of such a minimal model $\mathcal{M}$. For more details, see \cite[chapter 6]{Halperin}.

We shall inductively construct a minimal CDGA $\mathcal{M}_{n}$ ($n \in \mathbb{Z}$, $n \geq 0$) and a CDGA homomorphism $\varphi_{n}: \mathcal{M}_{n} \rightarrow \mathcal{A}$ such that $\varphi_{n}$ is $n$-regular and $(\mathcal{M}_{n}, \varphi_{n})$ is an extension of $(\mathcal{M}_{n-1}, \varphi_{n-1})$. Here $n$-regular means $\varphi_{n}^{*}: H^{i}(\mathcal{M}_{n}) \rightarrow H^{i}(\mathcal{A})$ is an isomorphism for $i \leq n$ and an injection for $i=n+1$. Given such $\mathcal{M}_n$ and $\varphi_n$, we define $\mathcal{M}= \bigcup_{n \geq 0} \mathcal{M}_{n}$, and $\varphi: \mathcal{M} \rightarrow \mathcal{A}$ by $\varphi|_{\mathcal{M}_{n}} = \varphi_{n}$. Certainly, $(\mathcal{M}, \varphi)$ is a minimal model of $\mathcal{A}$.

To construct $\mathcal{M}_n$ and $\varphi_n$, we start by putting $\mathcal{M}_{0} = \mathcal{M}_{0}^{0} = \mathbb{K}$,  and $\varphi_{0}: \mathcal{M}_{0}^{0} = \mathbb{K} \rightarrow \mathbb{K} \subseteq \mathcal{A}^{0}$ such that it is the identity isomorphism of $\mathbb{K}$. Then $\varphi_{0}$ is $0$-regular.

Next, assuming the existence of $(\mathcal{M}_{n}, \varphi_{n})$, we now construct $(\mathcal{M}_{n+1}, \varphi_{n+1})$. If $\varphi_{n}^{*}: H^{n+1}(\mathcal{M}_{n}) \rightarrow H^{n+1}(\mathcal{A})$ is surjective, then it is an isomorphism. Otherwise, choose a collection of cocycles $\{ \alpha_{0, \mu} \mid \mu \in \mathcal{I}_{0} \} \subseteq \mathcal{A}^{n+1}$ which represent a basis of the cokernel of $\varphi_{n}^{*}: H^{n+1}(\mathcal{M}_{n}) \rightarrow H^{n+1}(\mathcal{A})$. Here $\mathcal{I}_{0}$ is an ordered index set. Introducing free homogeneous generators $a_{0, \mu}$ of degree $n+1$ to $\mathcal{M}_{n}$, we get a CDGA $\mathcal{X}_{0}$. More precisely, $\mathcal{X}_{0} = \mathcal{M}_{n}\otimes \mathcal{F}$ where $\mathcal{F}$ is the free CDGA generated by these $a_{0, \mu}$ with zero differentials. We identify the new generators $1 \otimes a_{0, \mu} \in \mathcal{X}_{0}$ with $a_{0, \mu}$, and they succeed those old generators. Extend $\varphi_{n}$ to $\psi_{0}: \mathcal{X}_{0} \rightarrow \mathcal{A}$ such that $\psi_{0} (a_{0, \mu}) = \alpha_{0, \mu}$. Then $\psi_{0}^{*}: H^{i} (\mathcal{X}_{0}) \rightarrow H^{i}(\mathcal{A})$ is an isomorphism for $i \leq n+1$. Denote by $\mathcal{Y}_{0}$ the kernel of $\psi_{0}^{*}: H^{n+2}(\mathcal{X}_{0}) \rightarrow H^{n+2}(\mathcal{A})$. If $\mathcal{Y}_{0} \neq 0$, choose a collection of cocycles $\{ b_{1, \mu} \mid \mu \in \mathcal{I}_{1} \} \subseteq \mathcal{X}_{0}^{n+2}$ which represent a basis of $\mathcal{Y}_{0}$. Introducing free homogeneous generators $a_{1, \mu}$ of degree $n+1$ to $\mathcal{X}_{0}$, we get a CDGA extension $\mathcal{X}_{1}$ such that $d a_{1, \mu} = b_{1, \mu}$. Here the elements in $\mathcal{I}_{1}$ succeed those in $\mathcal{I}_{0}$. Since $\psi_{0} (b_{1, \mu})$ is exact in $\mathcal{A}$, we have $\psi_{0} (b_{1, \mu}) = d \alpha_{1, \mu}$ for some $\alpha_{1, \mu} \in \mathcal{A}^{n+1}$. Extend $\psi_{0}$ to $\psi_{1}: \mathcal{X}_{1} \rightarrow \mathcal{A}$ such that $\psi_{1} (a_{1, \mu}) = \alpha_{1, \mu}$. Then $\mathcal{Y}_{0}$ is killed in $H^{n+2}(\mathcal{X}_{1})$. Denote by $\mathcal{Y}_{1}$ the kernel of $\psi_{1}^{*}: H^{n+2}(\mathcal{X}_{1}) \rightarrow H^{n+2}(\mathcal{A})$. If $\mathcal{Y}_{1} \neq 0$, then we extend $(\mathcal{X}_{1}, \psi_{1})$ to be $(\mathcal{X}_{2}, \psi_{2})$ which kills $\mathcal{Y}_{1}$. Repeating this procedure, we get $\mathcal{M}_{n+1} = \bigcup_{p \geq 0} \mathcal{X}_{p}$ and $\varphi_{n+1}: \mathcal{M}_{n+1} \rightarrow \mathcal{A}$ such that $\varphi_{n+1}|_{\mathcal{X}_{p}} = \psi_{p}$. Clearly, $\varphi_{n+1}^{*}: H^{i}(\mathcal{M}_{n+1}) \rightarrow H^{i}(\mathcal{A})$ is an isomorphism for $i \leq n+1$. Suppose $b \in \mathcal{M}_{n+1}^{n+2}$ is a cocycle such that $\varphi_{n+1}(b)$ is exact in $\mathcal{A}$. This $b$ must be in $\mathcal{X}_{p}$ for some $p$. Then its cohomology class is killed in $\mathcal{X}_{p+1}$. Thus $\varphi_{n+1}^{*}: H^{n+2}(\mathcal{M}_{n+1}) \rightarrow H^{n+2}(\mathcal{A})$ is injective. We see that $\varphi_{n+1}$ is $(n+1)$-regular.

\begin{proof}[Proof of Proposition \ref{proposition_formal}]
Since $M(A) = S^{1} \times N$ and $S^{1}$ is formal, it suffices to show that $N$ is formal. We shall first construct a minimal model of $\Omega N$.

Since the matrix $A$ is diagonalizable, by Lemma \ref{lemma_odd_complex_homology}, $\dim H^{1}(\Omega N) = 1$ and for any $i$
\[
s_{2} \cup: H^{2i}(\Omega N) \rightarrow H^{2i+1}(\Omega N)
\]
is an isomorphism. Here $s_{2}$ is a basis of $H^{1}(\Omega N)$. Choose $\alpha \in (\Omega N)^{1}$ and $\beta_{j} \in (\Omega N)^{2}$ ($1 \leq j \leq k$) such that $[\alpha] = s_{2}$ and $\beta_{1}, \dots, \beta_{k}$ represent a basis of $H^{2}(\Omega N)$. Then $\alpha \wedge \beta_{1}, \dots, \alpha \wedge \beta_{k}$ represent a basis of $H^{3}(\Omega N)$.

Construct a minimal CDGA $\mathcal{M}_{2}$ which is generated by one homogeneous element $a$ of degree $1$ and homogeneous elements $b_{1}, \dots, b_{k}$ of degree $2$. Define the differential $d$ of $\mathcal{M}_2$ to be zero. Then $\mathcal{M}_{2}$ is the tensor product of the exterior algebra $\bigwedge(a)$ and the polynomial algebra $\bigwedge(b_{1}, \dots, b_{k})$. Since $(\bigwedge(a))^{i} = 0$ for $i>1$, the linear space $\mathcal{M}_{2}^{1}$ has a basis $a$, $\mathcal{M}_{2}^{2}$ has a basis $b_{1}, \dots, b_{k}$, and $\mathcal{M}_{2}^{3}$ has a basis $ab_{1}, \dots, ab_{k}$. Define $\varphi_{2}: \mathcal{M}_{2} \rightarrow \Omega N$ such that $\varphi_{2} (a) = \alpha$ and $\varphi_{2} (b_{j}) = \beta_{j}$. Then $\varphi_{2} (ab_{j}) = \alpha \wedge \beta_{j}$. We see that $\varphi_{2}^{*}: H^{i}(\mathcal{M}_{2}) \rightarrow H^{i}(\Omega N)$ is an isomorphism for $0 \leq i \leq 3$. In particular, $\varphi_{2}$ is $2$-regular.

Apply the argument before this proof, we can extend $\mathcal{M}_{2}$ to be a minimal model $\mathcal{M}\\$ for $\Omega N$. Note that the extension only introduces new generators of degrees greater than $2$. Therefore, $\mathcal{M}^{i} = \mathcal{M}_{2}^{i}$ for $0 \leq i \leq 2$. In particular, the differential $d$ vanishes on $\mathcal{M}^{i}$ for $0 \leq i \leq 2$. We further infer that $\mathcal{N}^{i} = \{ 0 \}$ for $0 \leq i \leq 2$, where $\mathcal{N}^{i}$ is defined as in Definition \ref{definition_s-formal}. Thus $N$ is $2$-formal.

Since $\dim N = 5$, by Theorem \ref{theorem_FM-(n-1)-formal}, we conclude that $N$ is formal.
\end{proof}

\section{Cohomology Jump Loci}\label{section_jump}
In this section, we will prove Theorem \ref{theorem_monodromy} and Theorem \ref{theorem_not_homotopy}. The proofs will be based on the theory of cohomology jump loci. We shall first review some basic aspects of this theory and the main result of \cite{wang}.

Let $X$ be a connected topological space that is homotopy equivalent to a finite CW complex. Obviously, $X$ is also path-connected. Define the character variety $\Char (X)= \homo (\pi_{1}(X), \mathbb{C}^{*})$ to be the set of rank one characters of $\pi_{1}(X)$, where $\mathbb{C}^{*} = \mathbb{C} \setminus \{0\}$. Then $\Char(X)$ is naturally isomorphic to the moduli space of rank one local systems on $X$. (See for example \cite{dimca} for the definition of local systems.) For each $\rho\in \Char (X)$, there is a unique rank one local system $L_{\rho}$ whose monodromy action is isomorphic to $\rho$. The character variety $\Char (X)$ is naturally an abelian complex linear algebraic group. In fact, since $\homo (\pi_{1}(X), \mathbb{C}^{*}) \cong \homo (H_{1}(X; \mathbb{Z}), \mathbb{C}^{*})$, we infer $\Char(X)$ is isomorphic to the product of $\left( \mathbb{C}^{*} \right)^{b_{1}(X)}$ and a finite abelian group, where $b_{1}(X)$ is the first Betti number of $X$.

The cohomology jump loci $\Sigma^j_k(X)$ of $X$ are defined by
\[
\Sigma^{j}_{k} (X)=\{\rho\in \Char (X)|\dim_{\mathbb{C}} H^{j}(X; L_{\rho}) \geq k\}.
\]
They are algebraic subsets of $\Char(X)$. It is easy to see that both $\Char(X)$ and $\Sigma^j_k(X)$ are homotopy invariants. More precisely, if $h: X \rightarrow Y$ is a homotopy equivalence, then $h$ induces an isomorphism $h^{*}: \Char (Y) \rightarrow \Char (X)$ of algebraic groups which maps $\Sigma^{j}_{k} (Y)$ onto $\Sigma^{j}_{k} (X)$.

\begin{theorem}\cite[theorem 1.3]{wang}\label{theorem_torsiontranslate}
Suppose $X$ is homotopy equivalent to a compact K\"ahler manifold. Then for any $j, k$, each irreducible component of $\Sigma^{j}_{k} (X)$ is of the form $\tau \cdot T$, where $\tau$ is a torsion element, i.e. an element of finite order, in $\Char (X)$ and $T$ is an irreducible linear subgroup of $\Char (X)$.
\end{theorem}
\begin{rmk}
In \cite{arapura}, it was proved by Arapura that each irreducible component of $\Sigma^{1}_{k} (X)$ is of the form $\tau \cdot T$, where $\tau$ is a unitary character. Using this result, Papadima and Suciu (\cite{papadima_suciu}) constructed a family of compact, orientable, formal 4-manifolds, which have the real homotopy type of smooth projective surfaces, but they are not homotopy equivalent to any compact K\"ahler manifold.
\end{rmk}

Let $p: E \rightarrow S^{1}$ be a fiber bundle with fiber $F$.  Suppose $E$ is connected and suppose E is homotopy equivalent to a finite CW complex. Then $p$ induces a homomorphism $p_{*}: \pi_{1}(E) \rightarrow \pi_{1}(S^{1})$ which further induces a homomorphism $p^{*}: \Char (S^{1}) \rightarrow \Char (E)$.

\begin{lemma}\label{lemma_jump_loci_injective}
Suppose $F$ is path-connected. Then $p^{*}: \Char (S^{1}) \rightarrow \Char (E)$ is injective.
\end{lemma}
\begin{proof}
Since $F$ is path-connected, $p_{*}: \pi_{1}(E) \rightarrow \pi_{1}(S^{1})$ is surjective, and hence $p^{*}: \Char (S^{1}) \rightarrow \Char (E)$ is injective.
\end{proof}

Fix an orientation of $S^1$. Then the monodromy action on $F$ is well-defined up to homotopy.

\begin{prop}\label{proposition_jump_equal}
Suppose that $H^{j}(F; \mathbb{C})$ is finite dimensional for all $j$. Denote by $\Gamma_{j}$ the set of eigenvalues of the monodromy action on $H^{j}(F; \mathbb{C})$. Define $\Gamma_{j}^{-1}= \{t^{-1} \mid t\in \Gamma_{j} \}$. Then
\[
\left( \bigcup_{j \geq 0} \Sigma^{j}_{1}(E) \right) \cap p^{*}\Char (S^{1})= p^{*} \left( \bigcup_{j \geq 0}\Gamma_{j}^{-1} \right).
\]
Here $\Char (S^{1})$ is identified with $\mathbb{C}^*$ via the fixed orientation of $S^1$.
\end{prop}

A more general form of Proposition \ref{proposition_jump_equal} is proved in \cite[theorem 3.6]{ps}. For reader's convenience, we give a proof here.

\begin{proof}[Proof of Proposition \ref{proposition_jump_equal}]
Given any $\hat{\rho}\in \Char (S^{1})$, denote by $\rho= p^{*}(\hat{\rho})$. We have the following Wang sequence with local system, which is a generalized version of Lemma \ref{lemma_wang}:
\[
\longrightarrow H^{j-1}(F; L_{\rho}) \overset{\delta}{\longrightarrow} H^{j}(E; L_{\rho}) \overset{i^{*}}{\longrightarrow} H^{j}(F; L_{\rho}) \overset{\eta^{*,j} - \hat{\rho}^{-1} \Id}{\longrightarrow} H^{j}(F; L_{\rho}) \longrightarrow,
\]
where the map $\eta^{*,j}: H^{j}(F; L_{\rho}) \rightarrow H^{j}(F; L_{\rho})$ is induced by the monodromy action. This result can be easily checked using the Mayer-Vietoris sequence, and it is very similar to the proposition 6.4.8 in \cite[p.\ 212]{dimca}.

By the above Wang sequence, we immediately get
\[
\dim_{\mathbb{C}} H^{j}(E; L_{\rho}) = \dim_{\mathbb{C}} \kn \left( \eta^{*,j} - \hat{\rho}^{-1} \Id \right) + \dim_{\mathbb{C}} \ckn \left( \eta^{*,j-1} - \hat{\rho}^{-1} \Id \right).
\]

By definition, $\rho \in \bigcup_{j \geq 0} \Sigma^{j}_{1}(E)$ if and only if $\dim_{\mathbb{C}} H^{j}(E; L_{\rho}) > 0$ for some $j$. By the above equation, $\dim_{\mathbb{C}} H^{j}(E; L_{\rho}) > 0$ if and only if $\hat{\rho}^{-1} \in \Gamma_{j}$ or $\hat{\rho}^{-1} \in \Gamma_{j-1}$. This completes the proof.
\end{proof}

\begin{proof}[Proof of Theorem \ref{theorem_monodromy}]
If $\lambda$ is an eigenvalue of this monodromy action, by Proposition \ref{proposition_jump_equal}, $p^{*} (\lambda^{-1}) \in \left( \bigcup_{j \geq 0} \Sigma^{j}_{1}(E) \right) \cap p^{*}\Char (S^{1})$ . By Lemma \ref{lemma_jump_loci_injective}, $p^{*}$ is injective. Thus it suffices to show that $\left( \bigcup_{j \geq 0} \Sigma^{j}_{1}(E) \right) \cap p^{*}\Char (S^{1})$ consists of torsion points, or equivalently $\Sigma^{j}_{1}(E)\cap p^{*}\Char (S^{1})$ consists of torsion points for any $j$.

By Theorem \ref{theorem_torsiontranslate}, an irreducible component of $\Sigma^{j}_{1}(E)$ is of the form $\tau \cdot T$, where $\tau$ is a torsion point. Suppose $V= (\tau \cdot T) \cap p^{*}\Char (S^{1})$ is nonempty. Then $V$ is an algebraic subset. It suffices further to verify that $V$ consists of torsion points.

By Proposition \ref{proposition_jump_equal}, $V$ is a countable set. Thus $V$ has dimension $0$, and hence contains finitely many points. Suppose $\tau$ has order $n$. Then $G = \bigcup_{0\leq r \leq n-1} \tau^{r} \cdot T$ is a subgroup, and hence so is $G \cap p^{*}\Char (S^{1})$. Let $\xi$ be a point in $V$. It is easy to see that $G \cap p^{*}\Char (S^{1}) = \bigcup_{0\leq r \leq n-1} \xi^{r} \cdot V$. So $G \cap p^{*}\Char (S^{1})$ is a finite group. Certainly, $V \subseteq G \cap p^{*}\Char (S^{1})$. Therefore, every element of $V$ is torsion.
\end{proof}

It remains to prove Theorem \ref{theorem_not_homotopy}.

\begin{lemma}\label{lemma_affine_subgroup}
Suppose $G$ is an affine group and $g \in G$ is not a torsion element. Then $g$ does not belong to two distinct irreducible affine subgroups with dimension $1$.
\end{lemma}
\begin{proof}
Assume $g$ belongs to two such subgroups $G_{1}$ and $G_{2}$. Then $G_{1} \cap G_{2}$ is an algebraic subset of $G_{1}$. Since $G_{1}$ and $G_{2}$ are irreducible, $G_{1} \neq G_{2}$ and $\dim (G_{1}) = \dim (G_{2}) =1$, we infer $\dim (G_{1} \cap G_{2}) =0$. Therefore, $G_{1} \cap G_{2}$ is a finite subgroup. Thus, $g \in G_{1} \cap G_{2}$ is a torsion element, which yields a contraction.
\end{proof}

\begin{lemma}\label{lemma_jump_loci_M(A)}
The manifold $M(A)$ in Theorem \ref{theorem_complex-symplectic} satisfies the following properties.
\begin{enumerate}
\item $\Char (M(A)) \cong \mathbb{C}^{*} \times \mathbb{C}^{*}$.
\item $\bigcup_{j \geq 0} \Sigma^{j}_{1} (M(A)) \subset \{ 1 \} \times \mathbb{C}^{*}$ under the above isomorphism.
\item Suppose $|\rho|\neq 1$. Then $(1, \rho) \in \bigcup_{j \geq 0} \Sigma^{j}_{1} (M(A))$ if and only if $\rho = |\lambda|^{2}$ or $\rho = |\lambda|^{-2}$, where $\lambda$ is an eigenvalue of $A$.
\end{enumerate}
\end{lemma}
\begin{proof}
By the assumption on $A$ in Theorem \ref{theorem_complex-symplectic}, $A$ has eigenvalues $\lambda$ and $\lambda^{-1}$ with $|\lambda|>1$.

According to (\ref{proposition_M(A)_1}), $M(A) = S^{1} \times N$, where $N$ is a fiber bundle over $S^{1}$ with fiber $K$, the Kummer surface. Denote by $p: M(A) \rightarrow S^{1} \times S^{1}$ the bundle projection. Since $K$ is simply connected (Proposition \ref{proposition_kummer_1-connected}), $p_{*}: \pi_{1} (M(A)) \rightarrow \pi_{1} (S^{1} \times S^{1})$ is an isomorphism (Proposition \ref{proposition_complex_Betti}), and hence
\[
\Char (M(A)) = \Char (S^{1} \times N) = p^{*} \Char (S^{1} \times S^{1}) \cong \mathbb{C}^{*} \times \mathbb{C}^{*}.
\]

By the fact that
\[
H^{0} (S^{1}; L_{\rho}) \cong H^{1} (S^{1}; L_{\rho}) \cong
\begin{cases}
\mathbb{C} & (\rho = 1) \\
0 & (\rho \neq 1)
\end{cases},
\]
and by the K\"{u}nneth formula, we infer that
\[
\bigcup_{j \geq 0} \Sigma^{j}_{1} (M(A)) = \left\{ (1, \rho) \middle| \rho \in \bigcup_{j \geq 0} \Sigma^{j}_{1} (N) \right\}.
\]

By (\ref{proposition_M(A)_1}), $A_{K}: K \rightarrow K$ in Proposition \ref{proposition_biholomorphic} is equal to the monodromy map of the bundle $N \rightarrow S^{1}$.
By Proposition \ref{proposition_jump_equal}, $\rho \in \bigcup_{j \geq 0} \Sigma^{j}_{1} (N)$ if and only if $\rho^{-1}$ is an eigenvalue of $A_{K}^{*}$ on $H^{*}(K; \mathbb{C})$. The conclusion now  follows immediately from Lemma \ref{lemma_kummer_odd_homology} and Lemma \ref{lemma_gA_eigenvalue}.
\end{proof}

\begin{proof}[Proof of Theorem \ref{theorem_not_homotopy}]
By assumption, $A_{i}$ has eigenvalues $\lambda_{i}$ and $\lambda_{i}^{-1}$ with $|\lambda_{i}| > 1$, where $i=1,2$. Since $A_{1}$ and $A_{2}$ have different spectral radii, $|\lambda_{1}| \neq |\lambda_{2}|$.

By Lemma \ref{lemma_jump_loci_M(A)}, $\Char (M(A_{i})) \cong \mathbb{C}^{*} \times \mathbb{C}^{*}$. In addition, $\bigcup_{j \geq 0} \Sigma^{j}_{1} (M(A_{i}))$ is contained in $\{ 1 \} \times \mathbb{C}^{*}$, and $(1, \rho) \in \bigcup_{j \geq 0} \Sigma^{j}_{1} (M(A_{i}))$ with $|\rho| \neq 1$ if and only if $\rho = |\lambda_{i}|^{2}$ or $\rho = |\lambda_{i}|^{-2}$.

We prove this theorem by contradiction. Assume $M(A_{1})$ and $M(A_{2})$ are of the same homotopy type. Since cohomology jump loci are homotopy invariants, we know that there is an isomorphism $\varphi: \mathbb{C}^{*} \times \mathbb{C}^{*} \rightarrow \mathbb{C}^{*} \times \mathbb{C}^{*}$ of affine groups such that
\[
\varphi \left( \bigcup_{j \geq 0} \Sigma^{j}_{1} (M(A_{1})) \right) = \bigcup_{j \geq 0} \Sigma^{j}_{1} (M(A_{2})).
\]

Therefore, by Lemma \ref{lemma_jump_loci_M(A)} (2),
\[
\bigcup_{j \geq 0} \Sigma^{j}_{1} (M(A_{2})) \subseteq \left( \{ 1 \} \times \mathbb{C}^{*} \right) \cap \varphi \left( \{ 1 \} \times \mathbb{C}^{*} \right).
\]
As $\bigcup_{j \geq 0} \Sigma^{j}_{1} (M(A_{2}))$ contains a non-torsion element $(1, |\lambda_{2}|^{2})$, by Lemma \ref{lemma_affine_subgroup},
\[
\varphi \left( \{ 1 \} \times \mathbb{C}^{*} \right) = \{ 1 \} \times \mathbb{C}^{*}.
\]
Now, $\varphi|_{\{ 1 \} \times \mathbb{C}^{*}}: \{ 1 \} \times \mathbb{C}^{*} \rightarrow \{ 1 \} \times \mathbb{C}^{*}$ is an isomorphism of linear groups. Thus, $\varphi (1, |\lambda_{1}|^{2}) = (1, |\lambda_{1}|^{2})$ or $\varphi (1, |\lambda_{1}|^{2}) = (1, |\lambda_{1}|^{-2})$.

Since $\varphi (1, |\lambda_{1}|^{2}) \in \bigcup_{j \geq 0} \Sigma^{j}_{1} (M(A_{2}))$ and $|\lambda_{1}|^{\pm 2} \neq 1$, by the above arguments, we have $|\lambda_{1}|^{2} = |\lambda_{2}|^{2}$ or $|\lambda_{1}|^{-2} = |\lambda_{2}|^{2}$. This contradicts the fact that $|\lambda_{1}| \neq |\lambda_{2}|$, $|\lambda_{1}| > 1$ and $|\lambda_{2}| > 1$.
\end{proof}

Motivated by the work of Voisin \cite{voisin} and Papadima-Suciu \cite{papadima_suciu}, we would like to end our paper with the following question.
\begin{question}
Does there exist a compact 6-manifold, which is of the same real homotopy type as a compact K\"ahler manifold and satisfies all the conclusions in Theorem \ref{theorem_complex-symplectic}?
\end{question}


\end{document}